\documentclass[sn-mathphys-num]{sn-jnl}

\usepackage{graphicx}%
\usepackage{multirow}%
\usepackage{amsmath,amssymb,amsfonts}%
\usepackage{amsthm}%
\usepackage{mathrsfs}%
\usepackage[title]{appendix}%
\usepackage{xcolor}%
\usepackage{textcomp}%
\usepackage{manyfoot}%
\usepackage{booktabs}%
\usepackage[linesnumbered,boxed,ruled,commentsnumbered]{algorithm2e}
\usepackage{algpseudocode}%
\usepackage{listings}%
\usepackage{xcolor}  

\theoremstyle{thmstyleone}%

\newtheorem{theorem}{Theorem}[section]
\newtheorem{property}{Property}[section]%
\newtheorem{lemma}{Lemma}[section]

\theoremstyle{thmstyletwo}%
\newtheorem{example}{Example}[section]%
\newtheorem{remark}{Remark}[section]%

\theoremstyle{thmstylethree}%
\newtheorem{definition}{Definition}[section]%

\begin{document}

\title[On the adaptive deterministic block coordinate descent methods with momentum for solving large linear least-squares problems]{On the adaptive deterministic block coordinate descent methods with momentum for solving large linear least-squares problems}


\author[1]{\fnm{Long-Ze} \sur{Tan}}\email{ba23001003@mail.ustc.edu.cn}

\author[2]{\fnm{Ming-Yu} \sur{Deng}}\email{my\_deng@stu.ecnu.edu.cn}
\equalcont{These authors contributed equally to this work.}

\author[3]{\fnm{Jia-Li} \sur{Qiu}}\email{qsq2110748243@stu.xjtu.edu.cn}
\equalcont{These authors contributed equally to this work.}

\author*[4]{\fnm{Xue-Ping} \sur{Guo}}\email{xpguo@math.ecnu.edu.cn}

\affil[1]{\orgdiv{School of Mathematical Sciences, Suzhou Institute for Advanced Research}, \orgname{University of Science and Technology of China}, \orgaddress{\city{Suzhou}, \postcode{215132}, \country{China}}}
\affil[2]{\orgdiv{School of International Trade and Economics}, \orgname{University of International Business and Economics}, \orgaddress{\city{Beijing}, \postcode{100029}, \country{China}}}
\affil[3]{\orgdiv{School of Mathematics and Statistics}, \orgname{Xi'an Jiaotong University}, \orgaddress{\city{Shaanxi}, \postcode{710049}, \country{China}}}
\affil*[4]{\orgdiv{Key Laboratory of MEA (Ministry of Education) \& Shanghai Key Laboratory of PMMP, School of Mathematical Sciences}, \orgname{East China Normal University}, \orgaddress{\city{Shanghai}, \postcode{200241}, \country{China}}}


\abstract{In this work, we first present an adaptive deterministic block coordinate descent method with momentum (mADBCD) to solve the linear least-squares problem, which is based on Polyak's heavy ball method and a new column selection criterion for a set of block-controlled indices defined by the Euclidean norm of the residual vector of the normal equation. The mADBCD method eliminates the need for pre-partitioning the column indexes of the coefficient matrix, and it also obviates the need to compute the  Moore-Penrose pseudoinverse of a column sub-matrix at each iteration. Moreover, we demonstrate the adaptability and flexibility in the automatic selection and updating of the block control index set. When the coefficient matrix has full rank, the theoretical analysis of the mADBCD method indicates that it linearly converges towards the unique solution of the linear least-squares problem. Furthermore, by effectively integrating count sketch technology with the mADBCD method, we also propose a novel count sketch adaptive block coordinate descent method with momentum (CS-mADBCD) for solving highly overdetermined linear least-squares problems and analysis its convergence. Finally, numerical experiments illustrate the advantages of the proposed two methods in terms of both CPU times and iteration counts compared to recent block coordinate descent methods.}

\keywords{Linear least-squares problem, Coordinate descent method, Block coordinate descent method, Heavy ball method, Convergence}


\maketitle

\section{Introduction}
It is well known that the least-squares problem has a rich background and applications in scientific computing and optimization, see \cite{EAmDis,EEMZ,EHWMAB,EST,EWWY} and the references therein. In this study, we consider the following large linear least-squares problems
\begin{equation}\label{linear least}
  \underset{\mathbf{x}\in\mathbb{R}^n}{\min}\|\mathbf{b}-\mathbf{A}\mathbf{x}\|_2,
\end{equation}
where coefficient matrix $\mathbf{A}\in\mathbb{R}^{m\times n}(m\geq n)$  is of full column rank, $\mathbf{b}\in\mathbb{R}^m$ is a real $m$-dimensional vector, and $\|\cdot\|_2$ denotes the Euclidean norm of either a vector or a matrix. In view of the fact that the coefficient matrix $\mathbf{A}$ is of full column rank, it follows that the least-squares solution $\mathbf{x}_{\star}$ to problem \eqref{linear least} exists and is unique. Furthermore, this solution can be mathematically expressed as $\mathbf{x}_{\star}=\mathbf{A}^{\dagger} \mathbf{b}$, with $\mathbf{A}^{\dagger}=(\mathbf{A}^T \mathbf{A})^{-1} \mathbf{A}^T$ and $\mathbf{A}^T$ being the Moore-Penrose pseudoinverse and the transpose of matrix $\mathbf{A}$, respectively. Certainly, the solution $\mathbf{x}_{\star}=\mathbf{A}^{\dagger} \mathbf{b}$  of problem \eqref{linear least} can be obtained by solving the normal equations
\begin{equation}\label{normal equations}
\mathbf{A}^T \mathbf{A} \mathbf{x}=\mathbf{A}^T \mathbf{b},
\end{equation}
as they are mathematically equivalent.
\subsection{The coordinate descent methods}
The problem \eqref{linear least} can be addressed through various direct methods, such as QR factorization and singular value decomposition (SVD) methods \cite{EBjo,EHinj}. However, when the matrix dimensions of problem \eqref{linear least} are extremely large, these direct methods typically require substantial storage and involve significant computational expenses, resulting in prolonged execution times during the solving process \cite{EBjo}. Therefore, numerous iterative methods have been developed for solving problem \eqref{linear least}. The coordinate descent (CD) method is one of the most economical and efficient iterative methods for solving linear least-squares problems, which is obtained by direct application of the classical Gauss-Seidel iterative method to the normal equations \eqref{normal equations} \cite{ERa}. Starting from an initial vector $\mathbf{x}^{(0)}$, the CD method can be formulated as
$$
\mathbf{x}^{(k+1)}=\mathbf{x}^{(k)}+\frac{\mathbf{A}_{(j_k)}^T(\mathbf{b}-\mathbf{A}\mathbf{x}^{(k)})}{\|\mathbf{A}_{(j_k)}\|_2^2} \mathbf{e}_{j_k}, \quad k=0,1,2, \cdots,
$$
where $\mathbf{A}_{(j_k)}$ denotes the $j_k$th column of the matrix $\mathbf{A}$ and $\mathbf{e}_{j_k}$ is the $j_k$th column of the identity matrix $\mathbf{I}\in\mathbb{R}^{n\times n}$.
\par
Leventhal and Lewis \cite{ELeLa} proposed the randomized coordinate descent (RCD) method which still maintained an expected linear convergence rate. The RCD method was also known as the randomized Gauss-Seidel (RGS) algorithm in some literatures on randomized algorithms for solving various large-scale linear systems \cite{EDu,EHeNR,EMaNR}. Due to the simplifications in the execution process, it has led to further development in both practical applications and theoretical research, see \cite{EHaJo,ELiLuX,ENeNeG,ENes} and the references therein.
\par
 In 2019, Bai and Wu designed an effective probability criterion for selecting the working columns from the coefficient matrix and proposed a greedy randomized coordinate descent (GRCD) method \cite{EBaWu}. By introducing a relaxation parameter $\theta\in[0,1]$ in the GRCD method, Zhang and Guo \cite{EZhG} proposed a relaxed greedy randomized coordinate descent (RGRCD) method and established its convergence analysis. Zhang and Li \cite{EZhLi} formulated a novel greedy Gauss-Seidel (GGS) method by computing the maximum entries of the residual vector associated with the normal equations \eqref{normal equations} to improve the GRCD method in 2020. In 2023, Tan and Guo \cite{ETaGo} proposed a multi-step greedy randomized coordinate descent (MGRCD) method with a faster convergence rate than the GRCD method.
\subsection{The block coordinate descent methods}
Firstly, for an arbitrary positive integer $z$, $[z]$ represents the set $\left\{1,2,\ldots,z\right\}$. Let $\left\{\mathcal{\tau}_1, \mathcal{\tau}_2, \ldots, \mathcal{\tau}_p\right\}$ denote a partition of the column indices of the coefficient matrix $\mathbf{A}$ such that, for $i, j=1,2, \ldots, p$ and $i \neq j$, it holds that $ \mathcal{\tau}_i \neq \emptyset$, $\mathcal{\tau}_i \cap \mathcal{\tau}_j=\emptyset$ and $\bigcup_{i=1}^p \mathcal{\tau}_i=[n]$. By extending the idea of the randomized block Kaczmarz method \cite{ENeeT,ENeZhZ}, Wu and Needell \cite{EWuNe} proposed a randomized block Gauss-Seidel (RBGS) method, which uses a randomized control method to select a column submatrix of the matrix $\mathbf{A}$ at each iteration. The RBGS method can be formulated as
\begin{equation}\label{equ:1.3}
  \begin{aligned}
  \mathbf{x}^{(k+1)} & =\mathbf{x}^{(k)}+\mathbf{I}_{\tau_{j_k}}(\mathbf{A}_{\tau_{j_k}}^T \mathbf{A}_{\tau_{j_k}})^{-1} \mathbf{A}_{\tau_{j_k}}^T(\mathbf{b}-\mathbf{A} \mathbf{x}^{(k)}) \\
  & =\mathbf{x}^{(k)}+\mathbf{I}_{\tau_{j_k}} \mathbf{A}_{\tau_{j_k}}^{\dagger}(\mathbf{b}-\mathbf{A} \mathbf{x}^{(k)}), \quad k=0,1,2, \ldots,
  \end{aligned}
\end{equation}
where the subscript $(\cdot)_{\tau_{j_k}}$ represents the column submatrix of the corresponding matrix that indexed by $\tau_{j_k}$, which lies in a column partition $\left\{\mathcal{\tau}_1, \mathcal{\tau}_2, \ldots, \mathcal{\tau}_p\right\}$ of the coefficient matrix.
\par
Li and Zhang \cite{ELiZ} proposed a greedy block Gauss-Seidel (GBGS) method that inherited the approach for selecting working columns from the RGRCD method. Liu et al. \cite{ELiJG} established a maximal residual block Gauss-Seidel (MRBGS) method and its convergence analysis was established. In 2022, Chen and Huang \cite{EChHu} proposed a fast block coordinate descent (FBCD) method using an iterative format that is different from the RBGS and MRBGS methods. The numerical experiments in  \cite{EChHu} demonstrated that, in comparison to the GBGS and MRBGS methods, the FBCD method exhibited superior performance in solving problem \eqref{linear least}.
\par
In this paper, we propose an adaptive deterministic block coordinate descent method with momentum, referred to as mADBCD, for approximating the solution to problem \eqref{linear least}. Our method utilizes a novel block control sequence and incorporates Polyak's heavy-ball method \cite{EPoBT}. At each step, the mADBCD method adaptively selects and flexibly updates the index set, and does not need to compute the Moore-Penrose pseudoinverse  of a column submatrix of the matrix $\mathbf{A}$. Under certain reasonable assumptions, theoretical analysis indicates that the mADBCD method will linearly converge to the least-squares solution $\mathbf{x}_{\star}$ of problem \eqref{linear least}. Numerical experiments demonstrate that the proposed mADBCD method exhibits superior numerical performance compared to the GBGS, MRBGS and FBCD methods, based on both iteration counts and CPU times.
\par
In addition, we introduce a count sketch adaptive deterministic block coordinate descent method with momentum (CS-mADBCD) to solve the highly overdetermined linear least-squares problem \eqref{linear least}, which linearly converges to the least-squares solution $\mathbf{x}_{\star}$ with probability $1-\delta$ with $\delta\in(0,1)$. Numerical experiments demonstrate that while the mADBCD method requires fewer iterations than the CS-mADBCD method, the latter consumes less CPU time than the former to achieve the desired accuracy. This observation highlights the superiority of the CS-mADBCD method over the mADBCD method in solving the extremely overdetermined linear least-squares problem \eqref{linear least}.
\par
The remainder of this paper is organized as follows. In Section \ref{sec:2}, we introduce the necessary notations, the FBCD method, the heavy ball method and two important lemmas, rspectively. The mADBCD method and its convergence analysis and error estimation are presented in detail in Section \ref{sec:3}.  In Section \ref{sec:4}, we present the CS-mADBCD method by combining the count sketch with the mADBCD method and establish its convergence analysis. In Section \ref{sec:5}, some numerical experiments are presented to investigate the effectiveness of the mADBCD and CS-mADBCD methods, showing their superiority over the GBGS, MRBGS and FBCD methods. Finally, we conclude this paper with some conclusions in Section \ref{sec:6}.

\section{Necessary notation and preliminary results}\label{sec:2}
\textbf{Notation.} For any vector $\mathbf{a}\in\mathbb{R}^n$, $\mathbf{a}_i$ and $\|\mathbf{x}\|_2$ represent its $i$th entry and the Euclidean $2$-norm, respectively. We use $\mathbf{e}_j$ to represent the $n$-dimensional unit coordinate vector with the $j$th element being 1 and all other components being 0. $\langle\mathbf{x}, \mathbf{y}\rangle$ denotes the Euclidean inner product of two $n$-dimensional column vectors $\mathbf{x}$ and $\mathbf{y}$. For a given matrix $\mathbf{G}\in\mathbb{R}^{m\times n}$, we let $\mathbf{G}_{(j)}, \mathbf{G}^{T}, \mathbf{G}^{\dagger}$, $\mathcal{R}(\mathbf{G})$ and $\|\mathbf{G}\|_F$ to represent  its $j$-th column, the transpose, the Moore-Penrose pseudoinverse, the column space and the Frobenius norm, respectively. In addition, the SVD of the matrix $\mathbf{G}$ is represented as
$
\mathbf{G}=\mathbf{U}\left[\begin{array}{c}
\mathbf{\Sigma} \\
    \mathbf{0}
\end{array}\right] \mathbf{V}^{T},
$
where $\mathbf{U} \in\mathbb{R}^{m\times m}$ and $\mathbf{V}\in\mathbb{R}^{n\times n}$  are orthogonal matrices and $\mathbf{\Sigma}=\operatorname{diag}\left(\sigma_{1}(\mathbf{G}), \sigma_{2}(\mathbf{G}), \ldots, \sigma_{n}(\mathbf{G})\right) \in \mathbb{R}^{n \times n}$. The nonzero singular values of $\mathbf{G}$ are $\sigma_{\max}(\mathbf{G}):=\sigma_1(\mathbf{G}) \geq \sigma_2(\mathbf{G}) \geq$ $\ldots \geq \sigma_r(\mathbf{G}):=\sigma_{\min }(\mathbf{G})>0$, where $r$ is the rank of $\mathbf{G}$ and $\sigma_{\max}(\mathbf{G})$ and $\sigma_{\min }(\mathbf{G})$ are used to represent the largest and smallest singular values of the matrix $\mathbf{G}$, respectively. For any vector $\mathbf{x}\in\mathbf{R}^n$, we denote by $\|\mathbf{x}\|_{\mathbf{G}}=\sqrt{\mathbf{x}^T \mathbf{G} \mathbf{x}}$ the energy norm of a symmetric positive-definite matrix $\mathbf{G} \in \mathbb{R}^{n \times n}$. Let the residual vector $\mathbf{r}^{(k)}=\mathbf{b}-\mathbf{A} \mathbf{x}^{(k)}$ and $\mathbf{s}^{(k)}=\mathbf{A}^T \mathbf{r}^{(k)}$. For an index set $\mathcal{\tau} \subseteq[n]$, $\mathbf{A}_{\mathcal{\tau}}$ denotes the column submatrix indexed by $\mathcal{\mathcal{\tau}}$ and  $|\mathcal{\mathcal{\tau}}|$ represents the cardinality of the set $\mathcal{\mathcal{\tau}} \subseteq[n]$.
\par
\textbf{The FBCD method.} Based on the greedy criterion of column selection in \cite{EBaWu}, Chen and Huang \cite{EChHu} proposed the following FBCD method for solving problem \eqref{linear least}.
\par
\begin{algorithm}[!htbp]
  \caption{The FBCD method}\label{Alg:FBCD}
  \KwIn {$\mathbf{A},\mathbf{b},\ell, \mathbf{x}^{(0)}\in \mathbb{R}^n$ and $\mathbf{s}^{(0)}=\mathbf{A}^T(\mathbf{b}-\mathbf{A}\mathbf{x}^{(0)})$}
	\KwOut {Approximate solution {$\mathbf{x}^{(\ell)}$}}
  \For{$k=0, 1, \cdots,\ell-1$}
	{Compute
  $$
\delta_k=\frac{1}{2}\left(\frac{1}{\|\mathbf{s}^{(k)}\|_2^2} \max _{1 \leq j \leq n}\left\{\frac{|(\mathbf{s}^{(k)})_{j}|^2}{\|\mathbf{A}_{(j)}\|_2^2}\right\}+\frac{1}{\|\mathbf{A}\|_F^2}\right).
$$
  \\
Determine the control index set
$$
\tau_k=\left\{j \mid | (\mathbf{s}^{(k)})_{j}|^2 \geq \delta_k \|\mathbf{s}^{(k)}\|_2^2 \|\mathbf{A}_{(j)}\|_2^2\right\}.
$$
Compute
$$
\mathbf{\eta}_k=\sum_{j \in \tau_k} (\mathbf{s}^{(k)})_{j} \mathbf{e}_j.
$$
  Set
  \begin{equation}\label{FBCD iteration}
    \mathbf{x}^{(k+1)}=\mathbf{x}^{(k)}+\frac{\mathbf{\eta}_k^T\mathbf{s}^{(k)}}{\|\mathbf{A}\mathbf{\eta}_k\|^2_2}\mathbf{\mathbf{\eta}}_k.
  \end{equation}
	}
\end{algorithm}
\par

\textbf{The heavy ball method.} Consider the global minimum of the optimization problem
$$
\min _{\mathbf{x} \in \mathbb{R}^n} f(\mathbf{x}),
$$
where $f(\mathbf{x})$ is a differentiable convex function.
The following gradient descent (GD) algorithm is one of the most commonly used methods, starting with an arbitrary point $\mathbf{x}^{(0)}$,
$$
\mathbf{x}^{(k+1)}=\mathbf{x}^{(k)}-\alpha_k \nabla f(\mathbf{x}^{(k)}), k=0,1,\cdots,
$$
where step-size $\alpha_k>0$ and $\nabla f(\mathbf{x}^{(k)})$ denotes the gradient of $f(\mathbf{x})$ at $\mathbf{x}^{(k)}$.
\par
Polyak \cite{EPoBT} modified the GD method by introducing a momentum term, denoted as $\beta(\mathbf{x}^{(k)}-\mathbf{x}^{(k-1)})$, commonly referred to as the heavy ball term.
This gives rise to the gradient descent method with momentum (mGD), which is commonly known as the heavy ball method:
\begin{equation}\label{equ:heavy ball method}
  \mathbf{x}^{(k+1)}=\mathbf{x}^{(k)}-\alpha_k \nabla f(\mathbf{x}^{(k)})+\beta(\mathbf{x}^{(k)}-\mathbf{x}^{(k-1)}), k=0,1,\cdots.
\end{equation}.
\par
Finally, we conclude this section with the following two lemmas, which are crucial for the convergence analysis of the mADBCD method.
\begin{lemma}\label{lemma 3.1}
  If $\mathbf{G} \in \mathbb{R}^{m \times n}$ has full column rank, we have $ \sigma_{\min}^2(\mathbf{G})\|\mathbf{x}\|_2^2\leq\|\mathbf{G x}\|_2^2 \leq \sigma_{\max}^2(\mathbf{G})\|\mathbf{x}\|_2^2$, $\forall \mathbf{x} \in \mathbb{R}^n$.
\end{lemma}
\begin{proof}
  From the SVD of the matrix, we have
  $$
  \mathbf{G}^{T} \mathbf{G}=\mathbf{V} \operatorname{diag}\left(\sigma_{1}^2(\mathbf{G}), \sigma_{2}^2(\mathbf{G}), \ldots, \sigma_{n}^2(\mathbf{G})\right) \mathbf{V}^{T}.
  $$
 Let $\mathbf{y}=\mathbf{V}^T\mathbf{x}=(y_1,y_2,\ldots,y_n)^T\in\mathbb{R}^n$, we have $\|\mathbf{y}\|_2=\|\mathbf{x}\|_2$. Since $\mathbf{V}$ is a orthogonal matrix, thus
 $$
 \begin{aligned}
  \|\mathbf{G}\mathbf{x}\|^2_2= \mathbf{x}^T(\mathbf{G}^T\mathbf{G})\mathbf{x}& =\mathbf{x}^T\mathbf{V} \operatorname{diag}\left(\sigma_{1}^2(\mathbf{G}), \sigma_{2}^2(\mathbf{G}), \ldots, \sigma_{n}^2(\mathbf{G})\right) \mathbf{V}^{T}\mathbf{x}\\
  & =\sigma_1^2(\mathbf{G})y_1^2+\sigma_2^2(\mathbf{G})y_2^2+\cdots+\sigma_n^2(\mathbf{G})y_n^2\\
  & \geq \sigma_n^2(\mathbf{G})\|\mathbf{y}\|_2^2=\sigma_{\min}^2(\mathbf{G})\|\mathbf{x}\|_2^2.
 \end{aligned}
 $$
Similarly, $\|\mathbf{G x}\|_2^2 \leq \sigma_{\max}^2(\mathbf{G})\|\mathbf{x}\|_2^2$ can be proved. Therefore, the conclusion of the theorem holds.
\end{proof}
\begin{lemma}(\cite{EHX})\label{lemma 3.2}
  Fix $F^{(1)}=F^{(0)} \geq 0$, and let $\left\{F^{(k)}\right\}_{k \geq 0}$ be a sequence of nonnegative real numbers satisfying the relation
  $$
    F^{(k+1)} \leq \gamma_1 F^{(k)}+\gamma_2 F^{(k-1)}, \quad \forall k \geq 1,
  $$
  where $\gamma_2 \geq 0, \gamma_1+\gamma_2<1$. Then the sequence satisfies the relation
$$
  F^{(k+1)} \leq q^k(1+\tau) {F}^0, \quad \forall k \geq 0,
$$
  where
$$
  q=\left\{\begin{array}{ll}
    \frac{\gamma_1+\sqrt{\gamma_1^2+4 \gamma_2}}{2}, & \text { if } \gamma_2>0 ; \\
    \gamma_1, & \text { if } \gamma_2=0,
    \end{array} \quad \right.
$$
and $\tau=q-\gamma_1 \geq 0$.
  Moreover, $\gamma_1+\gamma_2 \leq q<1$, with equality if and only if $\gamma_2=0$.
\end{lemma}

\section{The mADBCD method and its convergence}\label{sec:3}
In this section, we first present the mADBCD method. Subsequently, our focus will be on establishing its convergence analysis.
\par
To construct the mADBCD method, we need to introduce the following novel block control index set
$$
  \tau_k=\left\{\ j_{k}\mid |(\mathbf{s}^{(k)})_{j_k}|^{2} \geq \frac{\|\mathbf{s}^{(k)}\|_2^2}{n}\right\},
$$
where $\mathbf{s}^{(k)}=\mathbf{A}^T(\mathbf{b}-\mathbf{A}\mathbf{x}^{(k)})$. It is easy to know that the set $\tau_k$ is a non-empty set.
To expedite the convergence speed of the iterative scheme \eqref{FBCD iteration}, we apply the idea of the heavy ball method \eqref{equ:heavy ball method} to obtain the following mADBCD method.
\par
\begin{algorithm}[!htbp]
  \caption{The mADBCD method}\label{Alg:mADBCD}
  \KwIn {$\mathbf{A},\mathbf{b},\ell,\beta \geq 0, \mathbf{x}^{(0)}=\mathbf{x}^{(1)} \in \mathbb{R}^n$ and $\mathbf{s}^{(1)}=\mathbf{A}^T(\mathbf{b}-\mathbf{A}\mathbf{x}^{(1)})$.}
	\KwOut {Approximate solution {$\mathbf{x}^{(\ell)}$}.}
  \For{$k=1, 2, \cdots,\ell-1$}
	{Determine the control index set
  \begin{equation}\label{mADBCD set}
    \tau_k=\left\{\ j_{k}\mid |(\mathbf{s}^{(k)})_{j_k}|^{2} \geq \frac{\|\mathbf{s}^{(k)}\|_2^2}{n}\right\}.
  \end{equation}
  \\
  Compute
  \begin{equation}\label{mADBCD eta}
    \mathbf{\eta}_k = \sum\limits_{j_k\in \tau_k}(\mathbf{s}^{(k)})_{j_k}\mathbf{e}_{j_k}.
  \end{equation}
  Set
  \begin{equation}\label{mADBCD iteration}
    \mathbf{x}^{(k+1)}=\mathbf{x}^{(k)}+\frac{\mathbf{\eta}_k^T\mathbf{s}^{(k)}}{\|\mathbf{A}\mathbf{\eta}_k\|^2_2}\mathbf{\eta}_k + \beta(\mathbf{x}^{(k)}-\mathbf{x}^{(k-1)}).
  \end{equation}
	}
\end{algorithm}
As follows is the important property for the mADBCD method.
\begin{property}\label{property 2.1}
  From any initial vector $\mathbf{x}^{(0)}$, the iterative sequence $\left\{\mathbf{x}^{(k)}\right\}_{k=0}^{\infty}$ generated by the mADBCD method is well-defined.
\end{property}
\begin{proof}
  For any $\mathbf{x}^{(0)}=\mathbf{x}^{(1)}\in\mathbb{R}^{n}$, $\|\mathbf{A}^T(\mathbf{b}-\mathbf{A}\mathbf{x}^{(1)})\|_2$ either equals zero or not. If $\|\mathbf{A}^T(\mathbf{b}-\mathbf{A}\mathbf{x}^{(1)})\|_2=0$, the sequence contains only $\mathbf{x}^{(0)}$ and $\mathbf{x}^{(1)}$. When $\|\mathbf{A}^T(\mathbf{b}-\mathbf{A}\mathbf{x}^{(1)})\|_2\neq 0$, it is easy to show that the set $\tau_1$ defined by \eqref{mADBCD set} is non-empty.  In order to demostrate the existence of $\mathbf{x}^{(2)}$ defined by \eqref{mADBCD iteration} exists,  we  just need to show that $\|\mathbf{A}\mathbf{\eta}_1\|_2\neq 0$. Assuming $\|\mathbf{A}\mathbf{\eta}_1\|_2=0$, then $\mathbf{A}\mathbf{\eta}_1 =\mathbf{0}$ and
  \begin{equation}\label{equation:2.4}
    \mathbf{\eta}_1\mathbf{s}^{(1)}  = \mathbf{\eta}_1^T\mathbf{A}^T(\mathbf{b}-\mathbf{A}\mathbf{x}^{(1)})= (\mathbf{A}\mathbf{\eta}_1)^T(\mathbf{b}-\mathbf{A}\mathbf{x}^{(1)}) = 0.
  \end{equation}
  However, according to \eqref{mADBCD set} and \eqref{mADBCD eta}, it follows that
  \begin{equation}\label{equation:2.5}
    \mathbf{\eta}_1\mathbf{s}^{(1)}=\sum\limits_{j_1\in \tau_1}|(\mathbf{s}^{(1)})_{j_1}|^2 \geq \frac{1}{n}|\tau_1|\|\mathbf{s}^{(1)}\|^2_2 > 0,
  \end{equation}
  which generates a contradiction to \eqref{equation:2.4}, thus this shows the existence of $\mathbf{x}^{(2)}$.
  \par
  Suppose $\mathbf{x}^{(k)}$ ($k\geq 2$) has been computed by the mADBCD method, then we can repeat the above derivation process by using $\mathbf{\eta}_k$ and $\mathbf{x}^{(k)}$ to replace $\mathbf{\eta}_1$ and $\mathbf{x}^{(1)}$, respectively. It follows by induction that the iterative sequence $\left\{\mathbf{x}\right\}_{k\geq 0}$ generated by the mADBCD method is well-defined.
\end{proof}
For the convergence analysis and error estimation of the mADBCD method, we can establish the following theorem.
\begin{theorem}\label{theo:3.1}
  Consider the large linear least-squares problem $\underset{\mathbf{x}\in\mathbb{R}^n}{\min}\|\mathbf{b}-\mathbf{A}\mathbf{x}\|_2$, where $\mathbf{A}\in \mathbb{R}^{m \times n}(m \geq n)$ is of full column rank and $\mathbf{b}\in\mathbb{R}^m$ is a given vector. Suppose that the following expressions
  $$
  \gamma_1=1 + 3\beta + 2\beta^2  - (3\beta+1)\frac{|\tau_k|\sigma^2_{\min}(\mathbf{A})}{n \sigma^2_{\max}( \mathbf{A}_{\tau_k})}  \quad \text{and} \quad \gamma_2= 2\beta^2+\beta
  $$
  satisfy $\gamma_1+\gamma_2<1$.
  Then the iteration sequence $\left\{\mathbf{x}^{(k)}\right\}_{k=0}^{\infty}$, generated by the mADBCD method starting from any initial guess $\mathbf{x}^{(0)}$, exists and converges to the unique least-squares solution $\mathbf{x}_{\star}=\mathbf{A}^{\dagger}\mathbf{b}$, with the error estimate
  \begin{equation}\label{inequality:2.6}
    \|\mathbf{x}^{(k+1)}-\mathbf{x}_{\star}\|^2_{\mathbf{A^T\mathbf{A}}} \leq q^k(1+\tau)\|\mathbf{x}^{(0)}-\mathbf{x}_{\star}\|^2_{\mathbf{A^T\mathbf{A}}},
   \end{equation}
   where $q=\frac{\gamma_1+\sqrt{\gamma_1^2+4 \gamma_2}}{2}$ and $\tau=q-\gamma_1$. Moreover, $\gamma_1+\gamma_2 \leq q<1$.
\end{theorem}
\begin{proof}
  Based on the Property \ref{property 2.1}, it can be deduced that the iterative sequence $\left\{\mathbf{x}^{(k)}\right\}_{k\geq 0}$ indeed exists. For certain values of $ 0 \leq k \leq \infty$, when $\|\mathbf{A}^T(\mathbf{b}-\mathbf{A}\mathbf{x}^{(k)})\|_2=0$, $\left\{\mathbf{x}^{(k)}\right\}_{k\geq 0}$ is simply a sequence containing only a finite number of elements, in which case the iteration terminates with $\mathbf{x}^{(k)}=\mathbf{x}_{\star}$ and the mADBCD method succeeds in finding a solution to the least-squares problem \eqref{linear least}. If, for any $k\geq 0$, it holds that $\|\mathbf{A}^T(\mathbf{b}-\mathbf{A}\mathbf{x}^{(k)})\|_2\neq 0$, then we will prove that the sequence $\left\{\mathbf{x}^{(k)}\right\}_{k=0}^{\infty}$ converges to $\mathbf{x}_{\star}$.
  \par
  For $k \geq 0$, set
  $$\mathbf{P}_k = \frac{\mathbf{A}\mathbf{\eta}_k \mathbf{\eta}_k^T \mathbf{A}^T}{\|\mathbf{A}\mathbf{\eta}_k\|_2^2},$$
   it is easy to prove that $\mathbf{P}_k$ satisfies $\mathbf{P}_k=\mathbf{P}_k^T$ and $\mathbf{P}_k^2=\mathbf{P}_k$.
  According to the definition of $\mathbf{P}_k$ and $\mathbf{A}^T\mathbf{A}\mathbf{x}_{\star}=\mathbf{A}^T\mathbf{b}$, it holds that
  \begin{equation}\label{equation:2.6}
    \begin{aligned}
      \frac{\mathbf{\eta}_k^T \mathbf{s}^{(k)}}{\|\mathbf{A}\mathbf{\eta}_k\|^2_2}\mathbf{A}\mathbf{\eta}_k&=\frac{\mathbf{\eta}_k^T(\mathbf{A}^T\mathbf{b}-\mathbf{A}^T\mathbf{A}\mathbf{x}^{(k)})}{\|\mathbf{A}\mathbf{\eta}_k\|^2_2}\mathbf{A}\mathbf{\eta}_k=\frac{\mathbf{\eta}_k^T(\mathbf{A}^T\mathbf{A}\mathbf{x}_{\star}-\mathbf{A}^T\mathbf{A}\mathbf{x}^{(k)})}{\|\mathbf{A}\mathbf{\eta}_k\|^2_2}\mathbf{A}\mathbf{\eta}_k\\
      &=\frac{\mathbf{\eta}_k^T\mathbf{A}^T(\mathbf{A}\mathbf{x}_{\star}-\mathbf{A}\mathbf{x}^{(k)})}{\|\mathbf{A}\mathbf{\eta}_k\|^2_2}\mathbf{A}\mathbf{\eta}_k= \frac{\mathbf{A}\mathbf{\eta}_k \mathbf{\eta}_k^T\mathbf{A}^T(\mathbf{A}\mathbf{x}_{\star}-\mathbf{A}\mathbf{x}^{(k)})}{\|\mathbf{A}\mathbf{\eta}_k\|^2_2}\\
      &=-\mathbf{P}_k(\mathbf{A}(\mathbf{x}^{(k)}-\mathbf{x}_{\star})).
    \end{aligned}
  \end{equation}
 Then we can deduce that
  \begin{equation}\label{equation:2.7}
    \begin{aligned}
      &\|\mathbf{x}^{(k+1)}-\mathbf{x}_{\star}\|^2_{\mathbf{A}^T\mathbf{A}}=\|\mathbf{A}(\mathbf{x}^{(k+1)}-\mathbf{x}_{\star})\|^2_2\\
      =&\|\mathbf{A}(\mathbf{x}^{(k)}-\mathbf{x}_{\star})-\mathbf{P}_k(\mathbf{A}(\mathbf{x}^{(k)}-\mathbf{x}_{\star}))+\beta\mathbf{A}(\mathbf{x}^{(k)}-\mathbf{x}^{(k-1)})\|^2_2\\
      =&\|(\mathbf{I}-\mathbf{P}_k)(\mathbf{A}(\mathbf{x}^{(k)}-\mathbf{x}_{\star})) + \beta\mathbf{A}(\mathbf{x}^{(k)}-\mathbf{x}^{(k-1)})\|_2^2\\
      =&\|(\mathbf{I}-\mathbf{P}_k)(\mathbf{A}(\mathbf{x}^{(k)}-\mathbf{x}_{\star}))\|^2_2 + \beta^2 \|\mathbf{A}(\mathbf{x}^{(k)}-\mathbf{x}^{(k-1)})\|_2^2+2\beta\left\langle(\mathbf{I}-\mathbf{P}_k)(\mathbf{A}(\mathbf{x}^{(k)}-\mathbf{x}_{\star})),\mathbf{A}(\mathbf{x}^{(k)}-\mathbf{x}^{(k-1)})\right\rangle \\
      =&\|(\mathbf{I}-\mathbf{P}_k)(\mathbf{A}(\mathbf{x}^{(k)}-\mathbf{x}_{\star}))\|^2_2 + \beta^2 \|\mathbf{A}(\mathbf{x}^{(k)}-\mathbf{x}^{(k-1)})\|_2^2 \\&+2\beta\left\langle \mathbf{A}(\mathbf{x}^{(k)}-\mathbf{x}_{\star}), \mathbf{A}(\mathbf{x}^{(k)}-\mathbf{x}^{(k-1)})\right\rangle-2\beta\left\langle \mathbf{P}_k(\mathbf{A}(\mathbf{x}^{(k)}-\mathbf{x}_{\star})), \mathbf{A}(\mathbf{x}^{(k)}-\mathbf{x}^{(k-1)})\right\rangle.
    \end{aligned}
  \end{equation}
  \par
  We now carefully analyze each term on the right-hand side of the last equation in \eqref{equation:2.7}. By using $\mathbf{P}_k = \mathbf{P}_k^T$ and $\mathbf{P}_k^2 = \mathbf{P}_k$, the first term satisfies
  \begin{equation}\label{equation:2.8}
    \begin{aligned}
      &\|(\mathbf{I}-\mathbf{P}_k)(\mathbf{A}(\mathbf{x}^{(k)}-\mathbf{x}_{\star}))\|^2_2\\
      =&\left\langle (\mathbf{I}-\mathbf{P}_k)(\mathbf{A}(\mathbf{x}^{(k)}-\mathbf{x}_{\star})),(\mathbf{I}-\mathbf{P}_k)(\mathbf{A}(\mathbf{x}^{(k)}-\mathbf{x}_{\star}))\right\rangle \\
      =&(\mathbf{A}(\mathbf{x}^{(k)}-\mathbf{x}_{\star}))^T(\mathbf{I}-\mathbf{P}_k)^T(\mathbf{I}-\mathbf{P}_k)(\mathbf{A}(\mathbf{x}^{(k)}-\mathbf{x}_{\star}))\\
      =&(\mathbf{A}(\mathbf{x}^{(k)}-\mathbf{x}_{\star}))^T(\mathbf{I}-\mathbf{P}_k)(\mathbf{A}(\mathbf{x}^{(k)}-\mathbf{x}_{\star}))\\
     =&(\mathbf{A}(\mathbf{x}^{(k)}-\mathbf{x}_{\star}))^T(\mathbf{A}(\mathbf{x}^{(k)}-\mathbf{x}_{\star}))-\left\langle\mathbf{A}(\mathbf{x}^{(k)}-\mathbf{x}_{\star}), \mathbf{P}_k(\mathbf{A}(\mathbf{x}^{(k)}-\mathbf{x}_{\star})) \right\rangle \\
      =&\|\mathbf{A}(\mathbf{x}^{(k)}-\mathbf{x}_{\star})\|^2_2-\left\langle\mathbf{A}(\mathbf{x}^{(k)}-\mathbf{x}_{\star}), \mathbf{P}_k(\mathbf{A}(\mathbf{x}^{(k)}-\mathbf{x}_{\star})) \right\rangle.
    \end{aligned}
  \end{equation}
  Then, by
  \begin{equation}\label{equation:2.9}
    \begin{aligned}
      &\left\langle\mathbf{A}(\mathbf{x}^{(k)}-\mathbf{x}_{\star}), \mathbf{P}_k(\mathbf{A}(\mathbf{x}^{(k)}-\mathbf{x}_{\star})) \right\rangle = (\mathbf{A}(\mathbf{x}^{(k)}-\mathbf{x}_{\star}))^T \mathbf{P}_k(\mathbf{A}(\mathbf{x}^{(k)}-\mathbf{x}_{\star}))\\
       = &(\mathbf{A}(\mathbf{x}^{(k)}-\mathbf{x}_{\star}))^T \mathbf{P}_k^2(\mathbf{A}(\mathbf{x}^{(k)}-\mathbf{x}_{\star})) = (\mathbf{A}(\mathbf{x}^{(k)}-\mathbf{x}_{\star}))^T \mathbf{P}_k^T \mathbf{P}_k(\mathbf{A}(\mathbf{x}^{(k)}-\mathbf{x}_{\star}))\\
      = &\left\langle \mathbf{P}_k\mathbf{A}(\mathbf{x}^{(k)}-\mathbf{x}_{\star}), \mathbf{P}_k\mathbf{A}(\mathbf{x}^{(k)}-\mathbf{x}_{\star})\right\rangle = \|\mathbf{P}_k(\mathbf{A}(\mathbf{x}^{(k)}-\mathbf{x}_{\star}))\|^2_2
    \end{aligned}
  \end{equation}
  and \eqref{equation:2.8}, it yields
  \begin{equation}\label{equation:2.10}
    \|(\mathbf{I}-\mathbf{P}_k)(\mathbf{A}(\mathbf{x}^{(k)}-\mathbf{x}_{\star}))\|^2_2 = \|\mathbf{A}(\mathbf{x}^{(k)}-\mathbf{x}_{\star})\|^2_2-\|\mathbf{P}_k(\mathbf{A}(\mathbf{x}^{(k)}-\mathbf{x}_{\star}))\|^2_2.
  \end{equation}
  For the second term, by Cauchy's inequality we can deduce that
  \begin{equation}\label{equation:2.11}
    \begin{aligned}
     & \beta^2\|\mathbf{A}(\mathbf{x}^{(k)}-\mathbf{x}^{(k-1)})\|^2_2=\beta^2\|\mathbf{A}(\mathbf{x}^{(k)}-\mathbf{x}_{\star})+ \mathbf{A}(\mathbf{x}_{\star} -\mathbf{x}^{(k-1)})\|^2_2\\
      =&\beta^2\|\mathbf{A}(\mathbf{x}^{(k)}-\mathbf{x}_{\star})\|^2_2 + 2\beta^2\left\langle \mathbf{A}(\mathbf{x}^{(k)}-\mathbf{x}_{\star}), \mathbf{A}(\mathbf{x}_{\star}-\mathbf{x}^{(k-1)})\right\rangle + \|\mathbf{A}(\mathbf{x}^{(k-1)}-\mathbf{x}_{\star})\|^2_2\\
      \leq& \beta ^2 \|\mathbf{A}(\mathbf{x}^{(k)}-\mathbf{x}_{\star})\|^2_2 + 2\beta^2\|\mathbf{A}(\mathbf{x}^{(k)}-\mathbf{x}_{\star})\|^2_2\|\mathbf{A}(\mathbf{x}^{(k-1)}-\mathbf{x}_{\star})\|^2_2 + \beta^2\|\mathbf{A}(\mathbf{x}^{(k-1)}-\mathbf{x}_{\star})\|^2_2\\
     \leq & 2\beta^2\|\mathbf{A}(\mathbf{x}^{(k)}-\mathbf{x}_{\star})\|^2_2 + 2\beta^2\|\mathbf{A}(\mathbf{x}^{(k-1)}-\mathbf{x}_{\star})\|^2_2.
    \end{aligned}
  \end{equation}
  The third term can be rewritten as
  \begin{equation}\label{equation:2.12}
    \begin{aligned}
      &2\beta\left\langle \mathbf{A}(\mathbf{x}^{(k)}-\mathbf{x}_{\star}), \mathbf{A}(\mathbf{x}^{(k)}-\mathbf{x}^{(k-1)})\right\rangle \\
      =&\beta\left\langle \mathbf{A}(\mathbf{x}^{(k)}-\mathbf{x}_{\star}), \mathbf{A}(\mathbf{x}^{(k)}-\mathbf{x}^{(k-1)})\right\rangle + \beta\left\langle \mathbf{A}(\mathbf{x}^{(k)}-\mathbf{x}_{\star}), \mathbf{A}(\mathbf{x}^{(k)}-\mathbf{x}^{(k-1)})\right\rangle\\
      =&\beta\left\langle \mathbf{A}(\mathbf{x}^{(k)}-\mathbf{x}_{\star}), \mathbf{A}(\mathbf{x}^{(k)}-\mathbf{x}_{\star} + \mathbf{x}_{\star}-\mathbf{x}^{(k-1)})\right\rangle\\
       &+ \beta\left\langle \mathbf{A}(\mathbf{x}^{(k)}-\mathbf{x}^{(k-1)}+\mathbf{x}^{(k-1)}-\mathbf{x}_{\star}), \mathbf{A}(\mathbf{x}^{(k)}-\mathbf{x}^{(k-1)})\right\rangle\\
      =&\beta(\|\mathbf{A}(\mathbf{x}^{(k)}-\mathbf{x}_{\star})\|_2^2 + \|\mathbf{A}(\mathbf{x}^{(k)}-\mathbf{x}^{(k-1)})\|_2^2-\|\mathbf{A}(\mathbf{x}^{(k-1)}-\mathbf{x}_{\star})\|_2^2).
    \end{aligned}
  \end{equation}
  For the last term, by \eqref{equation:2.10} and the inequality $\|\mathbf{a}+\mathbf{b}\|^2_2\leq 2\|\mathbf{a}\|^2_2 + 2\|\mathbf{b}\|^2_2$, we have the following estimate
  \begin{equation}\label{equation:2.13}
    \begin{aligned}
      &-2\beta\left\langle \mathbf{P}_k(\mathbf{A}(\mathbf{x}^{(k)}-\mathbf{x}_{\star})), \mathbf{A}(\mathbf{x}^{(k)}-\mathbf{x}^{(k-1)})\right\rangle\\
      =&\beta(\|\mathbf{A}(\mathbf{x}^{(k)}-\mathbf{x}^{(k-1)})-\mathbf{P}_k(\mathbf{A}(\mathbf{x}^{(k)}-\mathbf{x}_{\star}))\|^2-\|\mathbf{A}(\mathbf{x}^{(k)}-\mathbf{x}^{(k-1)})\|^2_2-\|\mathbf{P}_k(\mathbf{A}(\mathbf{x}^{(k)}-\mathbf{x}_{\star}))\|^2_2)\\
      =&\beta(\|(\mathbf{I}-\mathbf{P}_k)(\mathbf{A}(\mathbf{x}^{(k)}-\mathbf{x}_{\star}))-\mathbf{A}(\mathbf{x}^{(k-1)}-\mathbf{x}_{\star})\|^2_2-\|\mathbf{A}(\mathbf{x}^{(k)}-\mathbf{x}^{(k-1)})\|^2_2-\|\mathbf{P}_k(\mathbf{A}(\mathbf{x}^{(k)}-\mathbf{x}_{\star}))\|^2_2)\\
      \leq&\beta(2\|(\mathbf{I}-\mathbf{P}_k)(\mathbf{A}(\mathbf{x}^{(k)}-\mathbf{x}_{\star}))\|_2^2 + 2\|\mathbf{A}(\mathbf{x}^{(k-1)}-\mathbf{x}_{\star})\|^2_2-\|\mathbf{A}(\mathbf{x}^{(k)}-\mathbf{x}^{(k-1)})\|^2_2-\|\mathbf{P}_k(\mathbf{A}(\mathbf{x}^{(k)}-\mathbf{x}_{\star}))\|^2_2)\\
      =&\beta(2\|\mathbf{A}(\mathbf{x}^{(k)}-\mathbf{x}_{\star})\|^2_2-3\|\mathbf{P}_k(\mathbf{A}(\mathbf{x}^{(k)}-\mathbf{x}_{\star}))\|^2_2+2\|\mathbf{A}(\mathbf{x}^{(k-1)}-\mathbf{x}_{\star})\|^2_2-\|\mathbf{A}(\mathbf{x}^{(k)}-\mathbf{x}^{(k-1)})\|^2_2).
    \end{aligned}
  \end{equation}
  Replacing \eqref{equation:2.10}, \eqref{equation:2.11}, \eqref{equation:2.12} and \eqref{equation:2.13} into \eqref{equation:2.7}, we obtain the following inequality
  \begin{equation}\label{equ:2.14}
   \begin{aligned}
    \|\mathbf{x}^{(k+1)} - \mathbf{x}_{\star}\|_{\mathbf{A}^T\mathbf{A}}^2 \leq&(1 + 3\beta + 2\beta^2)\|\mathbf{A}(\mathbf{x}^{(k)}-\mathbf{x}_{\star})\|^2_2\\
    &+(2\beta^2+\beta)\|\mathbf{A}(\mathbf{x}^{(k-1)}-\mathbf{x}_{\star})\|^2_2\\
    &-(3\beta+1)\|\mathbf{P}_k(\mathbf{A}(\mathbf{x}^{(k)}-\mathbf{x}_{\star}))\|^2_2.
   \end{aligned}
\end{equation}
\par
We now establish the inequality relationship between $\|\mathbf{P}_k(\mathbf{A}(\mathbf{x}^{(k)}-\mathbf{x}_{\star}))\|^2_2$ and $\|\mathbf{A}(\mathbf{x}^{(k)}-\mathbf{x}_{\star})\|^2_2$ by detailed derivation.
\par
  Let $\mathbf{E}_k\in \mathbb{R}^{n \times |\tau_k|}$ denote a matrix whose columns in turn consist of all vectors $\mathbf{e}_{j_k}\in\mathbb{R}^n$ with $j_k\in \tau_k$, then $\mathbf{A}_{\tau_k}=\mathbf{A}\mathbf{E}_k$ and $\mathbf{E}_k^T\mathbf{E}_k=\mathbf{I}_{|\tau_k|}$, where $\mathbf{I}_{|\tau_k|} \in \mathbf{R}^{|\tau_k| \times |\tau_k|}$ is the identity matrix. Denote by $\hat{\eta}_k=\mathbf{E}_k^T\mathbf{\eta}_k$, we can get
\begin{equation}\label{equation:2.15}
 \begin{aligned}
  \|\hat{\eta}_k\|^2_2&=\|\mathbf{E}_k^T\mathbf{\eta}_k\|^2_2=(\mathbf{E}_k^T\mathbf{\eta}_k)^T(\mathbf{E}_k^T\mathbf{\eta}_k)=\mathbf{\eta}_k^T\mathbf{E}_k\mathbf{E}_k^T\mathbf{\eta}_k\\
  &=\mathbf{\eta}_k^T \mathbf{\eta}_k = \|\mathbf{\eta}_k\|^2_2 = \sum\limits_{j_k\in \tau_k}|(\mathbf{s}^{(k)})_{j_k}|^2.
 \end{aligned}
\end{equation}
On the one hand, based on Lemma \ref{lemma 3.1}, we can deduce
\begin{equation}\label{equation:2.16}
  \|\mathbf{A}_{\tau_k}\hat{\eta}_k\|^2_2=(\mathbf{A}_{\tau_k}\hat{\eta}_k)^T(\mathbf{A}_{\tau_k}\hat{\eta}_k)=\hat{\eta}_k^T\mathbf{A}_{\tau_k}^T \mathbf{A}_{\tau_k} \hat{\eta}_k \leq \sigma^2_{\max}(\mathbf{A}_{\tau_k})\|\hat{\eta}_k \|^2_2.
\end{equation}
While on the other hand, we have
\begin{equation}\label{equation:2.17}
  \begin{aligned}
    \|\mathbf{A}_{\tau_k}\hat{\eta}_k\|^2_2 &= \hat{\eta}_k^T\mathbf{A}_{\tau_k}^T \mathbf{A}_{\tau_k} \hat{\eta}_k = (\mathbf{E}_k^T\mathbf{\eta}_k)^T \mathbf{A}_{\tau_k}^T \mathbf{A}_{\tau_k} \mathbf{E}_k^T\mathbf{\eta}_k\\
&=(\mathbf{E}_k^T\mathbf{\eta}_k)^T \mathbf{E}_k^T \mathbf{A}^T \mathbf{A} \mathbf{E}_k \mathbf{E}_k^T\mathbf{\eta}_k = (\mathbf{E}_k \mathbf{E}_k^T\mathbf{\eta}_k)^T \mathbf{A}^T \mathbf{A} (\mathbf{E}_k \mathbf{E}_k^T\mathbf{\eta}_k)\\
& = \mathbf{\eta}_k^T \mathbf{A}^T \mathbf{A} \mathbf{\eta}_k = \|\mathbf{A} \mathbf{\eta}_k\|^2_2.
  \end{aligned}
\end{equation}
Hence, based on equations \eqref{equation:2.16} and \eqref{equation:2.17}, we can derive the following inequality
\begin{equation}\label{equ:2.18}
  \|\mathbf{A}\mathbf{\eta}_k\|^2_2\leq \sigma^2_{\max}(\mathbf{A}_{\tau_k})\|\hat{\eta}_k \|^2_2.
\end{equation}
From the definition of $\mathbf{\eta}_k$ in \eqref{mADBCD eta} as well as \eqref{equation:2.15}, it yields that
\begin{equation}\label{equ:2.19}
  \begin{aligned}
    \mathbf{\eta}_k^T \mathbf{s}^{(k)} & = \left(\sum\limits_{j_k\in \tau_k}(\mathbf{s}^{(k)})_{j_k}\mathbf{e}_{j_k}^T\right)\mathbf{s}^{(k)}= \sum\limits_{j_k\in \tau_k}\left((\mathbf{s}^{(k)})_{j_k}\mathbf{e}_{j_k}^T \mathbf{s}^{(k)}\right)\\
    &=\sum\limits_{j_k\in \tau_k} |(\mathbf{s}^{(k)})_{j_k}|^2 = \|\mathbf{\eta}_k\|^2_2 = \|\hat{\mathbf{\eta}}_k\|^2_2.
  \end{aligned}
\end{equation}
It can be inferred from Lemma \ref{lemma 3.1} that
\begin{equation}\label{equ:2.17}
\begin{aligned}
  \|\mathbf{s}^{(k)}\|^2_2&=\|\mathbf{A}^T(\mathbf{b}-\mathbf{A}\mathbf{x}^{(k)})\|^2_2=\|\mathbf{A}^T\mathbf{A}(\mathbf{x}_{\star}-\mathbf{x}^{(k)})\|^2_2\geq \sigma^2_{\min}(\mathbf{A}^T)\|\mathbf{A}(\mathbf{x}^{(k)}-\mathbf{x}_{\star})\|^2_2\\
  &= \sigma_{\min}(\mathbf{A}^T\mathbf{A})\|\mathbf{A}(\mathbf{x}^{(k)}-\mathbf{x}_{\star})\|^2_2 = \sigma^2_{\min}(\mathbf{A})\|\mathbf{A}(\mathbf{x}^{(k)}-\mathbf{x}_{\star})\|^2_2.
\end{aligned}
\end{equation}
According to \eqref{equ:2.17} and  the definition of the set $\tau_k$ in \eqref{mADBCD set}, it holds that
\begin{equation}\label{equ:2.21}
 \begin{aligned}
  \|\mathbf{\eta}_k\|^2_2&=\mathbf{\eta}_k^T \mathbf{s}^{(k)}= \sum\limits_{j_k\in \tau_k} |(\mathbf{s}^{(k)})_{j_k}|^2\\
  & \geq |\tau_k| \frac{\|\mathbf{s}^{(k)}\|^2_2}{n} \geq |\tau_k| \frac{\sigma_{\min}(\mathbf{A}^T\mathbf{A})}{n} \|\mathbf{A}(\mathbf{x}^{(k)}-\mathbf{x}_{\star})\|^2_2\\
  & = |\tau_k| \frac{\sigma^2_{\min}(\mathbf{A})}{n} \|\mathbf{A}(\mathbf{x}^{(k)}-\mathbf{x}_{\star})\|^2_2.
 \end{aligned}
\end{equation}
Therefore, from equations \eqref{equ:2.18}, \eqref{equ:2.19} and \eqref{equ:2.21}, we have
\begin{equation}\label{equ:2.22}
  \begin{aligned}
    \|\mathbf{P}_k(\mathbf{A}(\mathbf{x}^{(k)}-\mathbf{x}_{\star}))\|^2_2 &= \frac{|\mathbf{\eta}_k\mathbf{s}^{(k)}|^2}{\|\mathbf{A}\mathbf{\eta}_k\|^2_2} =  \frac{|\mathbf{\eta}_k\mathbf{s}^{(k)}|\|\hat{\eta}_k\|^2_2}{\|\mathbf{A}\mathbf{\eta}_k\|^2_2}\\ & \geq \frac{\|\mathbf{\eta}_k\|^2_2}{\sigma^2_{\max}(\mathbf{A}_{\tau_k})} \geq \frac{|\tau_k|\sigma^2_{\min}(\mathbf{A})}{n \sigma^2_{\max}(\mathbf{A}_{\tau_k})} \|\mathbf{A}(\mathbf{x}^{(k)}-\mathbf{x}_{\star})\|^2_2.
  \end{aligned}
\end{equation}
\par
By substituting \eqref{equ:2.22} into \eqref{equ:2.14}, the following error estimation equation
\begin{equation}\label{equ:2.23}
  \begin{aligned}
   \|\mathbf{x}^{(k+1)} - \mathbf{x}_{\star}\|_{\mathbf{A}^T\mathbf{A}}^2 \leq&(1 + 3\beta + 2\beta^2)\|\mathbf{A}(\mathbf{x}^{(k)}-\mathbf{x}_{\star})\|^2_2+(2\beta^2+\beta)\|\mathbf{A}(\mathbf{x}^{(k-1)}-\mathbf{x}_{\star})\|^2_2\\
   &-(3\beta+1)\|\mathbf{P}_k(\mathbf{A}(\mathbf{x}^{(k)}-\mathbf{x}_{\star}))\|^2_2\\
   \leq & \left(1 + 3\beta + 2\beta^2  - (3\beta+1)\frac{|\tau_k|\sigma^2_{\min}(\mathbf{A})}{n \sigma^2_{\max}(\mathbf{A}_{\tau_k})} \right) \|\mathbf{A}(\mathbf{x}^{(k)}-\mathbf{x}_{\star})\|^2_2\\
   & +(2\beta^2+\beta)\|\mathbf{A}(\mathbf{x}^{(k-1)}-\mathbf{x}_{\star})\|^2_2\\
   & = \gamma_1\|\mathbf{A}(\mathbf{x}^{(k)}-\mathbf{x}_{\star})\|^2_2 + \gamma_2 \|\mathbf{A}(\mathbf{x}^{(k-1)}-\mathbf{x}_{\star})\|^2_2\\
   & = \gamma_1 \|\mathbf{x}^{(k)}-\mathbf{x}_{\star}\|^2_{\mathbf{A}^T \mathbf{A}} + \gamma_2 \|\mathbf{x}^{(k-1)}-\mathbf{x}_{\star}\|^2_{\mathbf{A}^T \mathbf{A}}
  \end{aligned}
\end{equation}
can be derived. For any nonzero vector $\hat{\mathbf{\xi}}\in \mathbb{R}^{|\tau_k|}$, it follows that
\begin{equation}
  \begin{aligned}
    \sigma^2_{\min}(\mathbf{A})=\sigma_{\min}(\mathbf{A}^T\mathbf{A})&\leq \frac{(\mathbf{E}_k\hat{\mathbf{\xi}})^T\mathbf{A}^T\mathbf{A}(\mathbf{E}_k\hat{\mathbf{\xi}})}{(\mathbf{E}_k\hat{\mathbf{\xi}})^T(\mathbf{E}_k\hat{\mathbf{\xi}})}=\frac{\hat{\mathbf{\xi}}^T\mathbf{E}_k^T \mathbf{A}^T \mathbf{A}\mathbf{E}_k\hat{\mathbf{\xi}}}{\hat{\mathbf{\xi}}^T\mathbf{E}_k^T \mathbf{E}_k\hat{\mathbf{\xi}}}\\
    &= \frac{\hat{\mathbf{\xi}}^T\mathbf{A}_{\tau_k}^T \mathbf{A}_{\tau_k} \hat{\mathbf{\xi}}}{\hat{\mathbf{\xi}}^T\hat{\mathbf{\xi}}} \leq \sigma_{\max}(\mathbf{A}_{\tau_k}^T \mathbf{A}_{\tau_k})=\sigma^2_{\max}(\mathbf{A}_{\tau_k}).
  \end{aligned}
\end{equation}
Accordingly, $\frac{|\tau_k|\sigma^2_{\min}(\mathbf{A})}{n \sigma^2_{\max}( \mathbf{A}_{\tau_k})}\in(0,1)$ holds, which makes $\gamma_1\geq 0$. Let $F^{(k)}:=\|\mathbf{x}^{(k)}-\mathbf{x}_{\star}\|^2_{\mathbf{A^T\mathbf{A}}}$, then \eqref{equation:2.13} can be written as
$$
F^{(k+1)}\leq \gamma_1 F^{(k)} + \gamma_2 F^{(k-1)}.
$$
The range of values of $\beta$ shows that $\gamma_2 \geq 0$.  If $\gamma_2=0$, then it follows easily that $\beta = 0$ and $q=\gamma_1\geq 0$. Also because $\gamma_1 + \gamma_2<1$ holds by assumption, the conditions of Lemma \ref{lemma 3.2} are satisfied. It follows from the conclusion of Lemma \ref{lemma 3.2} that the error estimate \eqref{inequality:2.6} holds and it is also further shown that the iterative sequence $\left\{\mathbf{x}^{(k)}\right\}_{k=0}^{\infty}$ generated by the mADBCD method converges to $\mathbf{x}_{\star}$.
\end{proof}
\begin{remark}
We discuss the existence of $\beta$ in Theorem \ref{theo:3.1}, ensuring that the crucial condition $\gamma_1+\gamma_2<1$ holds. Let
$$
\alpha_k = \frac{|\tau_k|\sigma^2_{\min}(\mathbf{A})}{n \sigma^2_{\max}( \mathbf{A}_{\tau_k})},
$$ then the important condition $\gamma_1+\gamma_2<1$ can be equivalently rewritten as
$$
4\beta^2+(4-3\alpha_k)\beta-\alpha_k<0.
$$
Since $\Delta_k=(4-3\alpha_k)^2+16\geq 16$, the $\beta$ satisfying $\gamma_1 + \gamma_2<1$ exists, and its value is dependent on the iteration count $k$.
\end{remark}
\begin{remark}\label{remark:3.3}
  When $\beta=0$, basis on \eqref{equ:2.23} it can be easily deduced that
  $$
  \|\mathbf{x}^{(k+1)}-\mathbf{x}_{\star}\|^2_{\mathbf{A}^T\mathbf{A}}\leq \left(1-\frac{|\tau_k|\sigma^2_{\min}(\mathbf{A})}{n \sigma^2_{\max}( \mathbf{A}_{\tau_k})}\right)\|\mathbf{x}^{(k)}-\mathbf{x}_{\star}\|^2_{\mathbf{A}^T\mathbf{A}}.
  $$
\end{remark}
Theorem \ref{theo:3.1} indicates that the upper bound on the convergence rate of the mADBCD method is related to the momentum parameter $\beta$, the minimum singular value of the coefficient matrix of the linear least-squares problem \eqref{linear least}, and the maximum singular value of the selected column submatrix.
However, the mADBCD method may converge faster than this upper bound.

\section{The CS-mADBCD method and its convergence}\label{sec:4}
In this section, we provide the motivation for combining the count sketch technique \cite{EChCFC, EThZh} with the mADBCD method. Finally, we propose the CS-mADBCD method, which is derived by integrating the count sketch technique with the mADBCD method, and establish the convergence of the proposed method.
\par
The motivation for combining the mADBCD method with the count sketch technique primarily lies in its ability to reduce computational complexity by compressing the matrix to a smaller, approximate form, while maintaining computational accuracy and significantly improving efficiency. Additionally, integrating count sketch with mADBCD minimizes large matrix operations, substantially reducing the computational time required to solve large-scale linear least squares problems. Furthermore, incorporating count sketch enables the mADBCD method to scale effectively to high-dimensional data, thereby expanding its applicability and enhancing its ability to efficiently handle large-scale linear systems.
We now list the definition of count sketch matrix or sparse embedding matrix and provide the related CS-mADBCD method in Algorithm \ref{Alg:CS-mADBCD}.
\begin{definition}(\cite{EChCFC, EThZh})
A count sketch matrix is defined to be $\mathbf{S}=\Phi D \in$ $\mathbb{R}^{d \times m}$. Here, $D$ is an $m \times m$ random diagonal matrix with each diagonal entry independently chosen to be +1 or -1 with equal probability, and $\Phi \in\{0,1\}^{d \times m}$ is a $d \times m$ binary matrix with $\Phi_{h(i), i}=1$ and all remaining entries 0 , where $h:[m] \rightarrow[d]$ is a random map such that for each $i \in[m], h(i)=j$ with probability $1 / d$ for each $j \in[d]$.
\end{definition}
\begin{algorithm}[H]
  \caption{The CS-mADBCD method}\label{Alg:CS-mADBCD}
  \KwIn {$\mathbf{A}, \mathbf{b},\ell,\beta \geq 0$ and $\mathbf{x}^{(0)}=\mathbf{x}^{(1)} \in \mathbb{R}^n$}
	\KwOut {Approximate solution {$\mathbf{x}^{(\ell)}$}}
  {\bfseries Initial}: Create a count sketch matrix $\mathbf{S}\in\mathbb{R}^{d \times m }$ with $d<m$, and compute $\tilde{\mathbf{A}}=\mathbf{S}\mathbf{A}$,   $\tilde{\mathbf{b}}=\mathbf{S}\mathbf{b}$ and $ \tilde{\mathbf{s}}^{(1)}=\tilde{\mathbf{A}}^T(\tilde{\mathbf{b}}-\tilde{\mathbf{A}}\mathbf{x}^{(1)})$.\\
  \For{$k=1, 2, \cdots,\ell-1$}
	{Determine the control index set
  \begin{equation}\label{CS-mADBCD set}
    \tilde{\tau}_k=\left\{\ j_{k}\mid |(\tilde{\mathbf{s}}^{(k)})_{j_k}|^{2} \geq \frac{\|\tilde{\mathbf{s}}^{(k)}\|_2^2}{n}\right\}.
  \end{equation}
  \\
  Compute
  \begin{equation}\label{CS-mADBCD eta}
    \tilde{\mathbf{\eta}}_k = \sum\limits_{j_k\in \tilde{\tau}_k}(\tilde{\mathbf{s}}^{(k)})_{j_k}\mathbf{e}_{j_k}.
  \end{equation}
  Set
  \begin{equation}\label{CS-mADBCD iteration}
    \mathbf{x}^{(k+1)}=\mathbf{x}^{(k)}+\frac{\tilde{\tau}_k^T\tilde{\mathbf{s}}^{(k)}}{\|\tilde{\mathbf{A}}\tilde{\mathbf{\eta}}_k\|^2_2}\tilde{\mathbf{\eta}}_k + \beta(\mathbf{x}^{(k)}-\mathbf{x}^{(k-1)}).
  \end{equation}
	}
\end{algorithm}
Similarly, we also have the following property.
\begin{property}\label{property 4.1}
  Starting from any initial vector $\mathbf{x}^{(0)}$, the iterative sequence $\left\{\mathbf{x}^{(k)}\right\}_{k=0}^{\infty}$ generated by the CS-mADBCD method is well-defined.
\end{property}
In the following, we first present two important lemmas that will play a crucial role in the proof of the convergence analysis of the CS-mADBCD method.
\begin{lemma}(\cite{EWood})\label{lemma:4.1}
   If $\mathbf{S} \in \mathrm{R}^{\mathrm{d} \times \mathrm{m}}$ is a count sketch transform with $d=(n^2+n) /(\delta \epsilon^2)$, where $0<\delta, \epsilon<1$, then we have that
   $$
   (1-\epsilon)\|\mathbf{A}\mathbf{x}\|_2 \leq\|\mathbf{S} \mathbf{A} \mathbf{x}\|_2 \leq(1+\epsilon)\|\mathbf{A} \mathbf{x}\|_2 \quad \text { for all } \mathbf{x} \in R^n \text {, }
   $$
   and
   $$
(1-\epsilon) \sigma_i(\mathbf{A}) \leq \sigma_i(\mathbf{S} \mathbf{A}) \leq(1+\epsilon) \sigma_i(\mathbf{A}) \quad \text { for all } 1 \leq i \leq n
$$
hold with probability $1-\delta$.
\end{lemma}
\begin{lemma}(\cite{EZhLih})\label{lemma:4.2}
  Let $\mathbf{S} \in \mathrm{R}^{\mathrm{d} \times \mathrm{m}}$  be given as in Lemma \ref{lemma:4.1}. Then $\mathcal{R}(\mathbf{\mathbf{A}^T\mathbf{S}^T})=\mathcal{R}(\mathbf{A}^T)$ holds with probability $1-\delta$.
\end{lemma}
\begin{theorem}\label{theo:4.1}
  Consider the large linear least-squares problem $\underset{\mathbf{x}\in\mathbb{R}^n}{\min}\|\mathbf{b}-\mathbf{A}\mathbf{x}\|_2$, where $\mathbf{A}\in \mathbb{R}^{m \times n}(m \geq n)$ is of full column rank and $\mathbf{b}\in\mathbb{R}^m$ is a given vector. Assume that the following expressions
  $$
  \tilde{\gamma}_1=(1 + 3\beta + 2\beta^2)\frac{(1+\epsilon)^2}{(1-\epsilon)^2} - (3\beta+1)\frac{|\tilde{\tau}_k|\sigma^2_{\min}(\mathbf{A})}{n \sigma^2_{\max}( \mathbf{A}_{\tilde{\tau}_k})}  \quad \text{and} \quad \tilde{\gamma}_2= (2\beta^2+\beta)\frac{(1+\epsilon)^2}{(1-\epsilon)^2}
  $$
  satisfy $\tilde{\gamma}_1+\tilde{\gamma}_2<1$. Then the iteration sequence $\left\{\mathbf{x}^{(k)}\right\}_{k=0}^{\infty}$, generated by the CS-mADBCD method starting from any initial guess $\mathbf{x}^{(0)}$, exists and converges to the unique least-squares solution $\mathbf{x}_{\star}=\mathbf{A}^{\dagger}\mathbf{b}$ with the probability $1-\delta$, and we have the following error estimate
  \begin{equation}\label{ineq:4.4}
    \|\mathbf{x}^{(k+1)}-\mathbf{x}_{\star}\|^2_{\mathbf{A^T\mathbf{A}}} \leq \tilde{q}^k(1+\tau)\|\mathbf{x}^{(0)}-\mathbf{x}_{\star}\|^2_{\mathbf{A^T\mathbf{A}}},
   \end{equation}
   where $\tilde{q}=\frac{\tilde{\gamma}_1+\sqrt{\tilde{\gamma}_1^2+4 \tilde{\gamma}_2}}{2}$ and $\tau=\tilde{q}-\tilde{\gamma}_1$. Moreover, $\tilde{\gamma}_1+\tilde{\gamma}_2 \leq \tilde{q}<1$.
\end{theorem}
\begin{proof}
  From Lemma \ref{lemma:4.2}, it can be inferred that the CS-mADBCD method is equivalent to applying the mADBCD method to the sketched linear least-squares problem $\underset{\mathbf{x}\in\mathbb{R}^n}{\min}\|\tilde{\mathbf{b}}-\tilde{\mathbf{A}}\mathbf{x}\|_2$, which allows us to utilize the convergence results of the mADBCD method directly. Hence, based on \eqref{equ:2.23}, it holds that
  \begin{equation}\label{equ:4.4}
    \begin{aligned}
     \|\mathbf{x}^{(k+1)} - \mathbf{x}_{\star}\|_{\tilde{\mathbf{A}}^T\tilde{\mathbf{A}}}^2 \leq&\left(1 + 3\beta + 2\beta^2  - (3\beta+1)\frac{|\tilde{\tau}_k|\sigma^2_{\min}(\tilde{\mathbf{A}})}{n \sigma^2_{\max}(\tilde{\mathbf{A}}_{\tilde{\tau}_k})} \right) \|\tilde{\mathbf{A}}(\mathbf{x}^{(k)}-\mathbf{x}_{\star})\|^2_2\\
     & +(2\beta^2+\beta)\|\tilde{\mathbf{A}}(\mathbf{x}^{(k-1)}-\mathbf{x}_{\star})\|^2_2.
    \end{aligned}
  \end{equation}
  Furthermore, based on Lemma \ref{lemma:4.1}, with probability $1-\delta$, we can attain
  \begin{equation}\label{inequality:4.5}
    \begin{aligned}
      &\sigma^2_{\min}(\tilde{\mathbf{A}})\geq(1-\epsilon)^2\sigma^2_{\min}(\mathbf{A}),\\
      &\sigma^2_{\max}(\tilde{\mathbf{A}}_{\tilde{\tau}_k})\leq (1+\epsilon)^2\sigma_{\max}(\mathbf{A}_{\tilde{\tau}_k}),\\
     & \|\tilde{\mathbf{A}}(\mathbf{x}^{(k)}-\mathbf{x}_{\star})\|^2_2\leq (1 + \epsilon)^2\|\mathbf{A}(\mathbf{x}^{(k)}-\mathbf{x}_{\star})\|^2_2,\\
      &\|\tilde{\mathbf{A}}(\mathbf{x}^{(k-1)}-\mathbf{x}_{\star})\|^2_2\leq (1 + \epsilon)^2\|\mathbf{A}(\mathbf{x}^{(k-1)}-\mathbf{x}_{\star})\|^2_2,\\
     & (1-\epsilon)^2\|\mathbf{x}^{(k+1)}-\mathbf{x}_{\star}\|^2_{\mathbf{A}^T\mathbf{A}} \leq  \|\mathbf{x}^{(k+1)} - \mathbf{x}_{\star}\|_{\tilde{\mathbf{A}}^T\tilde{\mathbf{A}}}^2.
    \end{aligned}
  \end{equation}
  Thus, substituting the above inequalities \eqref{inequality:4.5} into \eqref{equ:4.4}, with probability $1-\delta$, we obtain that
  \begin{equation}\label{inequ:4.17}
    \begin{aligned}
      \|\mathbf{x}^{(k+1)}-\mathbf{x}_{\star}\|^2_{\mathbf{A}^T\mathbf{A}} &\leq \frac{1}{(1-\epsilon)^2}\|\mathbf{x}^{(k+1)} - \mathbf{x}_{\star}\|_{\tilde{\mathbf{A}}^T\tilde{\mathbf{A}}}^2\\
      & \leq\left[(1 + 3\beta + 2\beta^2)\frac{(1+\epsilon)^2}{(1-\epsilon)^2}  - (3\beta+1)\frac{|\tilde{\tau}_k|\sigma^2_{\min}(\mathbf{A})}{n \sigma^2_{\max}(\mathbf{A}_{\tilde{\tau}_k})} \right] \|\mathbf{A}(\mathbf{x}^{(k)}-\mathbf{x}_{\star})\|^2_2\\
      & +\frac{(1+\epsilon)^2}{(1-\epsilon)^2}(2\beta^2+\beta)\|\mathbf{A}(\mathbf{x}^{(k-1)}-\mathbf{x}_{\star})\|^2_2\\
      & = \tilde{\gamma}_1\|\mathbf{A}(\mathbf{x}^{(k)}-\mathbf{x}_{\star})\|^2_2 + \tilde{\gamma}_2 \|\mathbf{A}(\mathbf{x}^{(k-1)}-\mathbf{x}_{\star})\|^2_2\\
      & = \tilde{\gamma}_1 \|\mathbf{x}^{(k)}-\mathbf{x}_{\star}\|^2_{\mathbf{A}^T \mathbf{A}} + \tilde{\gamma}_2 \|\mathbf{x}^{(k-1)}-\mathbf{x}_{\star}\|^2_{\mathbf{A}^T \mathbf{A}}.
    \end{aligned}
  \end{equation}
  Let $F^{(k)}:=\|\mathbf{x}^{(k)}-\mathbf{x}_{\star}\|^2_{\mathbf{A^T\mathbf{A}}}$, then \eqref{inequ:4.17} can be written as
  $$
  F^{(k+1)}\leq \tilde{\gamma}_1 F^{(k)} + \tilde{\gamma}_2 F^{(k-1)}.
  $$
  The range of values of $\beta$ shows that $\tilde{\gamma}_2 \geq 0$.  If $\tilde{\gamma}_2=0$, then it follows easily that $\beta = 0$ and $q=\tilde{\gamma}_1\geq 0$. Also because $\tilde{\gamma}_1 + \tilde{\gamma}_2<1$ holds by assumption, the conditions of Lemma \ref{lemma 3.2} are satisfied. Thus, inequality \eqref{equ:4.4} is then similarly obtained. Hence the iterative sequence $\left\{\mathbf{x}\right\}_{k=0}^{\infty}$ generated by the CS-mADBCD method converges to $\mathbf{x}_{\star}$.
\end{proof}
When $\beta=0$, similarly to Remark \ref{remark:3.3}, we have
$$
\|\mathbf{x}^{(k+1)}-\mathbf{x}_{\star}\|^2_{\mathbf{A}^T\mathbf{A}} \leq\left(\frac{(1+\epsilon)^2}{(1-\epsilon)^2}-\frac{|\tilde{\tau}_k|\sigma^2_{\min}(\mathbf{A})}{n \sigma^2_{\max}(\mathbf{A}_{\tilde{\tau}_k})} \right) \|\mathbf{x}^{(k)}-\mathbf{x}_{\star}\|^2_{\mathbf{A}^T \mathbf{A}}.
$$

\section{Numerical experiments}\label{sec:5}
In this section, we present a series of numerical examples to demonstrate the superiority of the mADBCD and CS-mADBCD methods compared to recently developed block coordinate descent methods, including the GBGS \cite{ELiZ}, MRBGS \cite{ELiJG} and FBCD \cite{EChHu} methods. The numerical performances of these five methods is evaluated based on the number of iteration steps (denoted by ``IT'') and CPU time in seconds (denoted by ``CPU'').  Here CPU and IT mean the arithmetical averages of the elapsed CPU times and the required iteration steps to run the corresponding methods $10$ times. All the numerical experiments are performed on a computer with Intel(R) Core(TM) i7-10510U processor with 1.83 GHz and 16GB RAM, under the Windows 10 operating systems using MATLAB R2020b with machine precision $10^{-16}$.
\par
In this experiment, we consider solving the following five different linear least-squares problems.
\begin{example}\label{example 5.1}
Linear least-squares problems with their coefficient matrices being fifty dense matrices generated by using the MATLAB function \textbf{randn}$(m,n)$ and the right-hand side vectors $b$ being set as $\mathbf{b}=\mathbf{A}\mathbf{x}_{\star}$, where the solution $\mathbf{x}_{\star}$ is randomly generated by the function \textbf{randn}$(n,1)$.
 \end{example}

\begin{example}\label{example 5.2}
Linear least-squares problems with their coefficient matrices being twenty sparse matrices from the Florida sparse matrix collection \cite{EDaHu} with their dimensions, densities and condition numbers listed in Table \ref{tab:5} and the right-hand side vectors $b$ being set as \ref{example 5.1}.
\end{example}

\begin{example}\label{example 5.3}
 Image reconstruction problems from the $2$D seismic travel-time tomography by using the function \textbf{seismictomo}$(N, s, p)$ in the MATLAB AIR Tools package \cite{EHaJo} with $N=50, s=80$ and $p=120$. That means a sparse matrix with 9600 rows and 2500 columns is generated. The right-hand side vectors $b$ are set as $\mathbf{b}=\mathbf{A}\mathbf{x}_{\star}$ with the true solutions $\mathbf{x}_{\star}$ shown as images in Figure \ref{fig:20} (top left).
\end{example}

\begin{example}\label{example 5.4}
X-ray CT reconstruction (the Shepp-Logan medical model) problems generated by function \textbf{fancurvedtomo}$(N, \theta, P)$ in the MATLAB package AIR Tools II \cite{EHaJo} with $N=50, \theta=0: 1: 300^{\circ} and P=50$, which generate a sparse matrix with 15050 rows and 2500 columns. A corresponding true solution $\mathbf{x}_{\star}$ is shown as images in Figure \ref{fig:21} (top left) and $\mathbf{b}=\mathbf{A}\mathbf{x}_{\star}$.
\end{example}

\begin{example}\label{example 5.5}
Linear least-squares problems with their coefficient matrices being thirty sparse matrices generated by using the MATLAB function \textbf{sprandn}$(m,n,s)$ and the right-hand side vectors $b$ being set as \ref{example 5.1}. Here, $m$, $n$ and $s$ denote the numbers of row, column and density of the generated matrices, respectively.
\end{example}
All methods start with an initial vector $\mathbf{x}^{(0)}=\mathbf{0}$. The termination criterion for Examples \ref{example 5.1}, \ref{example 5.2} and \ref{example 5.5} is that the relative solution error (RSE) at the approximate solution $\mathbf{x}^{(k)}$ satisfies
\begin{equation}\label{termination}
  \text{RSE}=\frac{\|\mathbf{x}^{(k)}-\mathbf{x}_{\star}\|^2_2}{\|\mathbf{x}_{\star}\|^2_2}<10^{(-6)};
\end{equation}
while for the examples of image reconstruction (Examples \ref{example 5.3} and \ref{example 5.4}), it is based on the CPU time reaching 180 seconds.
\par
Both the GBGS and MRBGS methods require the multiplication of the Moore-Penrose pseudoinverse $\mathbf{A}_{\tau_k}^{\dagger}$ of matrix $\mathbf{A}_{\tau_k}$ with the vector $\mathbf{r}^{(k)}=\mathbf{b}-\mathbf{A}\mathbf{x}^{(k)}$ for each iteration. Therefore, to reduce computational loads, we utilize the MATLAB function \textbf{lsqminnorm}$(\mathbf{A}_{\tau_k}, \mathbf{r}^{(k)})$ to compute  $\mathbf{A}_{\tau_k}^{\dagger}\mathbf{r}^{(k)}$, instead of using the MATLAB function \textbf{pinv}$(\mathbf{A}_{\tau_k})$ to compute  $\mathbf{A}_{\tau_k}^{\dagger}$and then applied to $\mathbf{r}^{(k)}$. For the MRBGS method, the block control index set is determined by
$$
\tau_k=\left\{ j_k|| (\mathbf{s}^{(k)})_{(j_k)}|^2 \geq 0.3 \max _{1 \leq j \leq n}|s_k^{(j)}|^2\right\},
$$
similar to numerical experiments in \cite{ELiJG}. The value of the parameter $\theta$ in the GBGS method is chosen as in \cite{ELiZ}.
Through numerical experiments, we have observed that when the momentum parameter $\beta$ in the mADBCD method exceeds 0.9, it tends to diverge for the majority of numerical examples. Therefore, we will implement the mADBCD method with $\beta$ falling within the interval $[0, 0.9]$. The relatively optimal momentum parameters $\beta$ and $\beta_{CS}$ in Tables \ref{tab:10}-\ref{tab:13} are experimentally found by minimizing the numbers of iteration steps during the implementation of the mADBCD and CS-mADBCD methods, respectively.

\subsection{Comparison of the mADBCD method with the GBGS, MRBGS and FBCD methods}
In this subsection, we compare the numerical performances of the GBGS, MRBGS, FBCD and mADBCD methods on Examples \ref{example 5.1}-\ref{example 5.4}, see Tables \ref{tab:1}-\ref{tab:4} and \ref{tab:6}-\ref{tab:9}. To better compare the convergence speeds of these four methods conveniently, we also present the speed-up of the mADBCD method against the other methods, defined by
$$
\text { speed-up{\_}Method }=\frac{\text { CPU of Method }}{\text { CPU of mADBCD }}.
$$

\begin{figure}[!htbp]
  \renewcommand\figurename{Figure}
    \centering
  \begin{minipage}[t]{0.45\linewidth}
    \centering
    \includegraphics[scale=0.40]{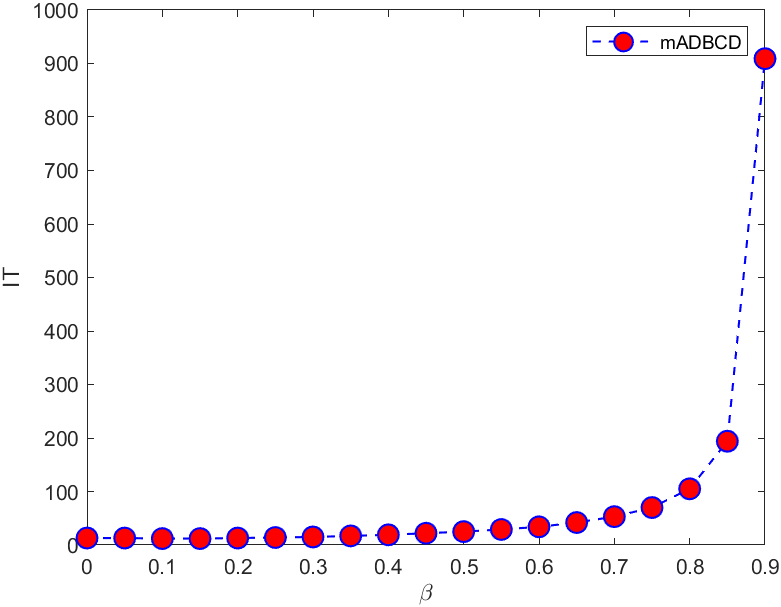}
  \end{minipage}
  \begin{minipage}[t]{0.45\linewidth}
    \centering
    \includegraphics[scale=0.40]{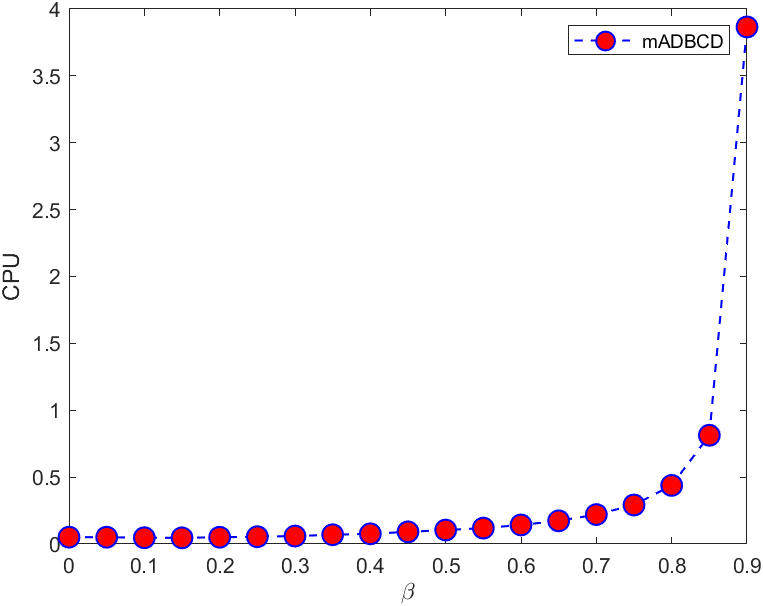}
  \end{minipage}
  \vspace{0.0cm}
  \caption{Pictures of $\beta$ versus IT (left) and CPU (right) for mADBCD when $\mathbf{A}$=\textbf{randn}$(7500,750)$.}
  \label{fig:1}
\end{figure}

Figures \ref{fig:1}-\ref{fig:5} illustrate the impact of varying momentum parameter $\beta$ on the iteration count and CPU time for the mADBCD method. From these five figures, we can deduce that when the ratio of rows $m$ to columns $n$ in the coefficient matrix is greater than or equal to $5$, as shown in Figures \ref{fig:1}-\ref{fig:2}, the mADBCD method exhibits minimal sensitivity to the selection of the momentum parameter $\beta$ when $\beta\leq 0.5$. Moreover, the optimal value of $\beta$ that minimizes the iteration count for this method typically falls within the range $[0, 0.3]$. This can also be inferred from the values of $\beta$ in Tables \ref{tab:1}-\ref{tab:2}. However, when $0.55\leq \beta \leq 0.9$, the iteration count and CPU time of the mADBCD method tend to increase as $\beta$ increases. When the ratio of $m$ to $n$ is less than or equal to $0.5$, as shown in Figures \ref{fig:3}-\ref{fig:5}, it can be observed that the choice of the momentum parameter $\beta$ significantly influences the numerical performances of the mADBCD method. In this case, as the value of $\beta$ increases, the number of iterations and CPU time of the method tend to decrease gradually and then increase. For solving such least-squares problems, the preferable range for the momentum parameter $\beta$ in the mADBCD method is between 0.5 and 0.65. This aspect is clearly demonstrated in Tables \ref{tab:3}-\ref{tab:4}.

\begin{figure}[!htbp]
  \renewcommand\figurename{Figure}
    \centering
  \begin{minipage}[t]{0.45\linewidth}
    \centering
    \includegraphics[scale=0.40]{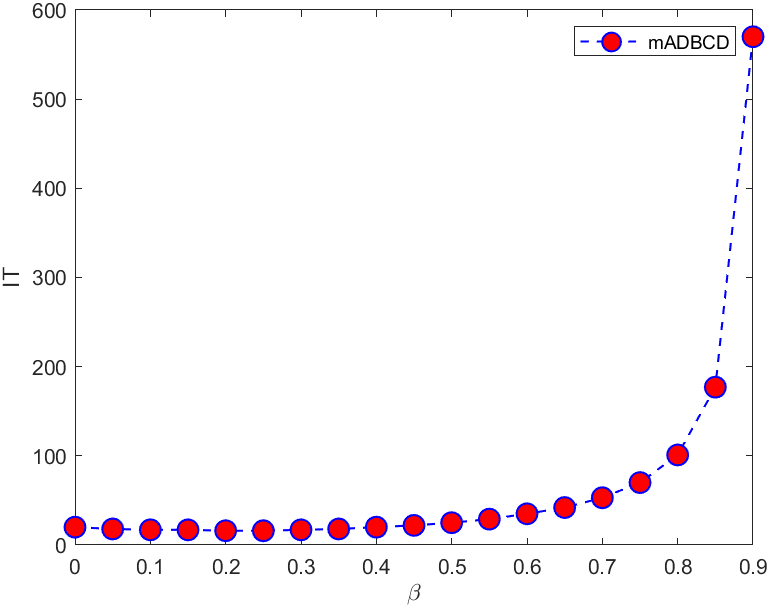}
  \end{minipage}
  \begin{minipage}[t]{0.45\linewidth}
    \centering
    \includegraphics[scale=0.40]{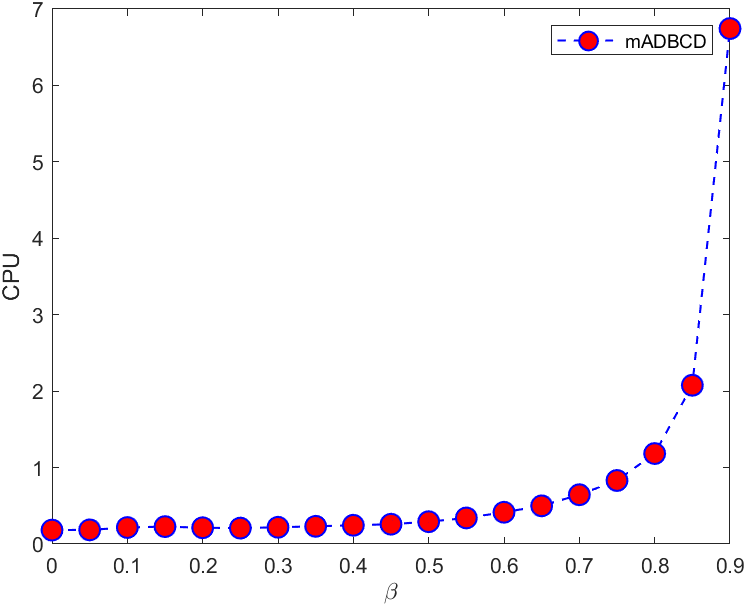}
  \end{minipage}
  \vspace{0.0cm}
  \caption{Pictures of $\beta$ versus IT (left) and CPU (right) for mADBCD when $\mathbf{A}$=\textbf{randn}$(7500,1500)$.}
  \label{fig:2}
\end{figure}
\begin{figure}[!htbp]
  \renewcommand\figurename{Figure}
    \centering
  \begin{minipage}[t]{0.45\linewidth}
    \centering
    \includegraphics[scale=0.40]{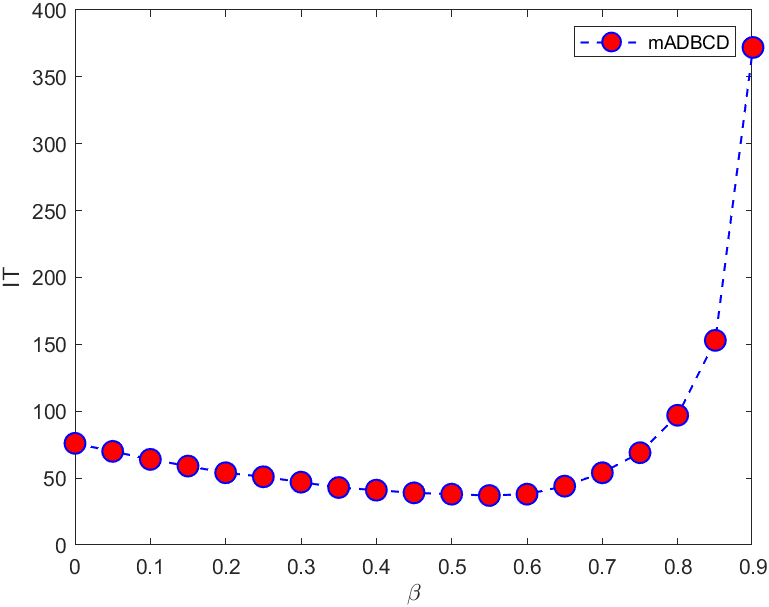}
  \end{minipage}
  \begin{minipage}[t]{0.45\linewidth}
    \centering
    \includegraphics[scale=0.40]{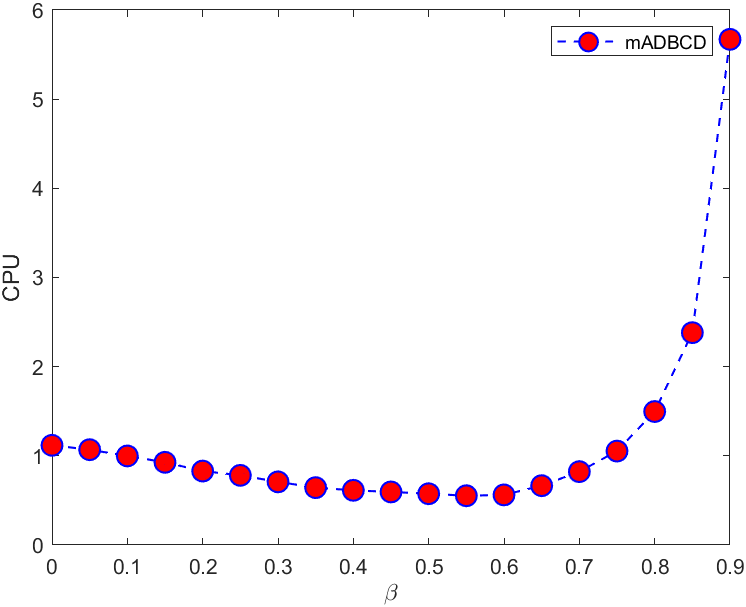}
  \end{minipage}
  \vspace{0.0cm}
  \caption{Pictures of $\beta$ versus IT (left) and CPU (right) for mADBCD when $\mathbf{A}$=\textbf{randn}$(6000,3000)$.}
  \label{fig:3}
\end{figure}
\begin{figure}[!htbp]
  \renewcommand\figurename{Figure}
    \centering
  \begin{minipage}[t]{0.45\linewidth}
    \centering
    \includegraphics[scale=0.40]{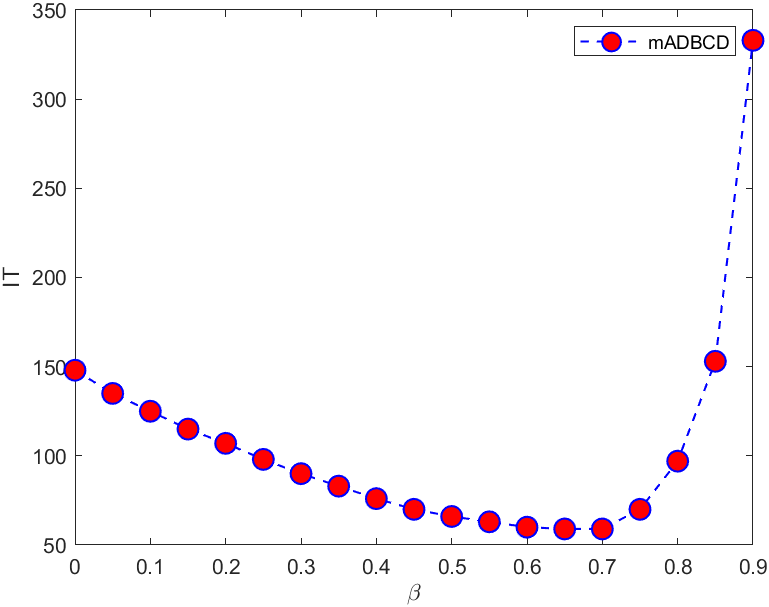}
  \end{minipage}
  \begin{minipage}[t]{0.45\linewidth}
    \centering
    \includegraphics[scale=0.40]{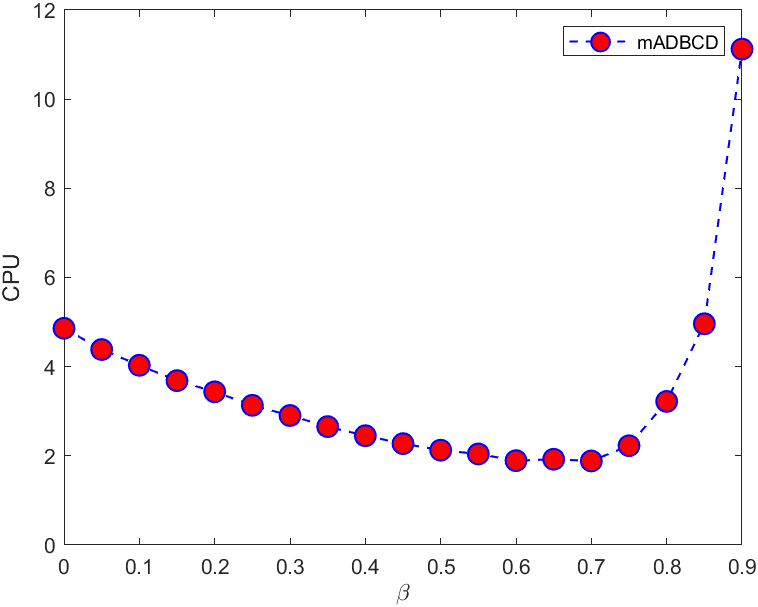}
  \end{minipage}
  \vspace{0.0cm}
  \caption{Pictures of $\beta$ versus IT (left) and CPU (right) for mADBCD when $\mathbf{A}$=\textbf{randn}$(8000,5000)$.}
  \label{fig:4}
\end{figure}
\begin{figure}[!htbp]
  \renewcommand\figurename{Figure}
    \centering
  \begin{minipage}[t]{0.45\linewidth}
    \centering
    \includegraphics[scale=0.40]{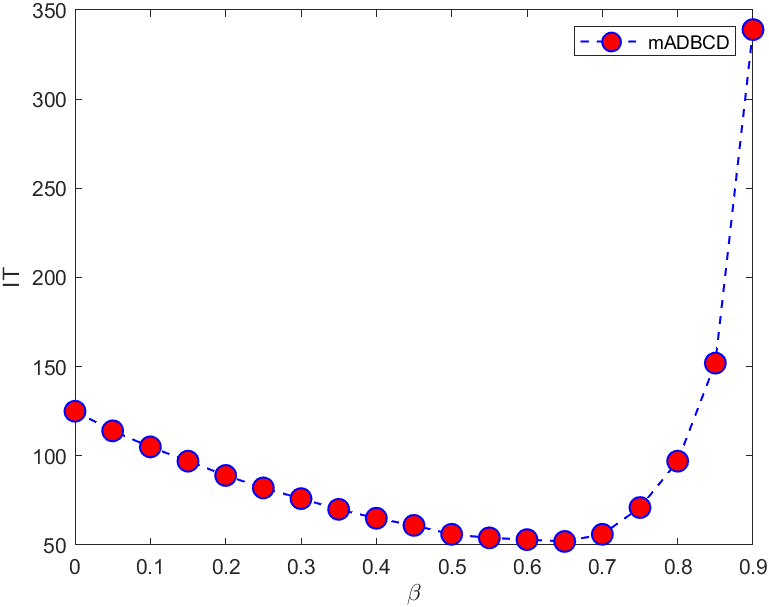}
  \end{minipage}
  \begin{minipage}[t]{0.45\linewidth}
    \centering
    \includegraphics[scale=0.40]{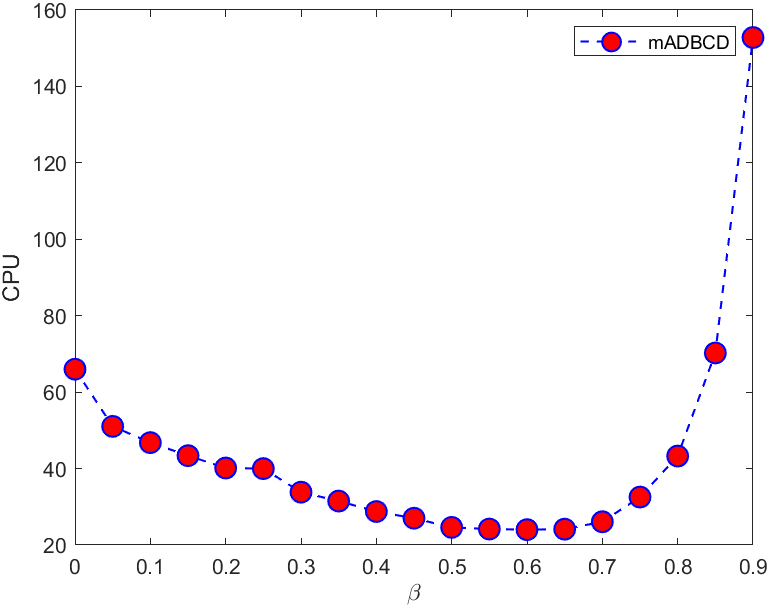}
  \end{minipage}
  \vspace{0.0cm}
  \caption{Pictures of $\beta$ versus IT (left) and CPU (right) for mADBCD when $\mathbf{A}$=\textbf{randn}$(21000,12500)$.}
  \label{fig:5}
\end{figure}

\begin{table}[!htbp]
  \centering
  \caption{IT and CPU of GBGS, MRBGS, FBCD and mADBCD for different $\mathbf{A}$=\textbf{randn}$(m,n)$.}\label{tab:1}%
  \begin{tabular*}{\hsize}{@{}@{\extracolsep{\fill}}lllllll@{}}
  \toprule
  $m$ & & $3500$ & $4500$ & $5500$ & $6500$ & $7500$ \\
  $n$ & & $350$ & $450$ & $550$ & $650$ & $750$ \\
  \midrule
  GBGS & IT & $47$ & $51$ & $49$ & $54$ & $52$ \\
  & CPU & $0.0896$ & $0.1476$ & $ 0.2191$ & $0.5432$ & $0.6154$ \\
  MRBGS& IT & $20$ & $22$ & $22$ & $22$ & $22$ \\
  & CPU & $0.0614$ & $0.1297$ & $0.1909$ & $0.3065$ & $0.4079$ \\
  FBCD & IT & $48$ & $53$ & $53$ & $53$ & $52$ \\
  & CPU & $ 0.0275$ & $0.0742$ & $ 0.1189$ & $0.1901$ & $0.3318$ \\
  mADBCD & $\beta$ & $0.10$ & $0.20$ & $0.10$ & $0.15$ & $0.15$ \\
  & IT & $12$ & $13$ & $12$ & $12$ & $12$ \\
  & CPU & $0.0055$ & $0.0163$ & $0.0246$ & $ 0.0351$ & $0.0486$ \\
  speed-up{\_GBGS}& & $ 16.29$ & $9.06$ & $8.91$ & $15.47$ & $12.66$  \\
  speed-up{\_MRBGS}& & $ 11.16$ & $7.96$ & $7.76$ & $8.73$ & $8.39$  \\
  speed-up{\_FBCD}& & $ 5.00$ & $4.55$ & $4.83$ & $5.42$ & $6.83$  \\
  \bottomrule
  \end{tabular*}
\end{table}

\begin{table}[!htbp]
  \centering
  \caption{IT and CPU of GBGS, MRBGS, FBCD and mADBCD for different $\mathbf{A}$=\textbf{randn}$(m,n)$.}\label{tab:2}%
  \begin{tabular*}{\hsize}{@{}@{\extracolsep{\fill}}lllllll@{}}
  \toprule
  $m$ & & $3500$ & $4500$ & $5500$ & $6500$ & $7500$ \\
  $n$ & & $700$ & $900$ & $1100$ & $1300$ & $1500$ \\
  \midrule
  GBGS & IT & $84$ & $89$ & $99$ & $94$ & $99$ \\
  & CPU & $0.4387$ & $0.5182$ & $1.2411$ & $2.2123$ & $3.3781$ \\
  MRBGS & IT & $32$ & $34$ & $36$ & $35$ & $37$ \\
  & CPU & $0.3575$ & $0.4770$ & $0.7900$ & $1.9296$ & $2.6369$ \\
  FBCD  & IT & $85$ & $96$ & $97$ & $95$ & $100$ \\
  & CPU & $0.1582$ & $0.3426$ & $0.5269$ & $ 1.0664$ & $1.3669$ \\
  mADBCD & $\beta$ & $0.25$ & $0.25$ & $0.25$ & $0.30$ & $0.25$ \\
  & IT & $16$ & $16$ & $16$ & $16$ & $16$ \\
  & CPU & $ 0.0256 $ & $0.0480$ & $0.0744$ & $0.1042$ & $  0.1436$ \\
  speed-up{\_GBGS}& & $ 17.14$ & $10.78$ & $16.68$ & $21.23$ & $23.52  $  \\
  speed-up{\_MRBGS}& & $13.96$ & $9.94$ & $10.62$ & $18.52$ & $ 18.36$  \\
  speed-up{\_FBCD}& & $6.18$ & $ 7.14$ & $7.08$ & $10.23$ & $9.52$  \\
  \bottomrule
  \end{tabular*}
\end{table}

\begin{table}[!htbp]
  \centering
  \caption{IT and CPU of GBGS, MRBGS, FBCD and mADBCD for different $\mathbf{A}$=\textbf{randn}$(m,n)$.}\label{tab:3}%
  \begin{tabular*}{\hsize}{@{}@{\extracolsep{\fill}}lllllll@{}}
  \toprule
  $m$ & & $4000$ & $5000$ & $6000$ & $7000$ & $8000$ \\
  $n$ & & $1000$ & $2000$ & $3000$ & $4000$ & $5000$ \\
  \midrule
  GBGS & IT & $112$ & $ 283$ & $508$ & $ 768$ & $1164$ \\
  & CPU & $0.5404$ & $5.6765$ & $12.4093$ & $59.3751$ & $141.3525$ \\
  MRBGS& IT & $42$ & $78$ & $ 129$ &  $190$ & $ 277$ \\
  & CPU & $0.4224$ & $3.6299$ & $23.7199$ & $100.0796$ & $277.5756$ \\
  FBCD & IT & $122$ & $285$ & $505$ & $783$ & $ 1166 $ \\
  & CPU & $0.3635$ & $ 2.2697$ & $7.4719$ & $36.4584$ & $ 79.3965$ \\
  mADBCD & $\beta$ & $0.30$ & $0.45$ & $0.55$ & $0.65$ & $0.65$ \\
  & IT & $18$ & $27$ & $37$ & $47$ & $59$ \\
  & CPU & $0.0516$ & $0.2069$ & $0.5545$ & $1.0630$ & $1.9208 $ \\
  speed-up{\_GBGS}& & $10.47$ & $ 27.44$ & $22.38$ & $55.86$ & $73.59$  \\
  speed-up{\_MRBGS}& & $8.19$ & $ 17.54$ & $42.78$ & $94.15$ & $ 144.51$  \\
  speed-up{\_FBCD}& & $ 7.04$ & $10.97$ & $ 13.48$ & $34.30$ & $ 41.34$  \\
  \bottomrule
  \end{tabular*}
\end{table}

\begin{table}[!htbp]
  \centering
  \caption{IT and CPU of GBGS, MRBGS, FBCD and mADBCD for different $\mathbf{A}$=\textbf{randn}$(m,n)$.}\label{tab:4}%
  \begin{tabular*}{\hsize}{@{}@{\extracolsep{\fill}}lllllll@{}}
  \toprule
  $m$ & & $10000$ & $13000$ & $16000$ & $19000$ & $21000$ \\
  $n$ & & $5000$ & $6500$ & $8500$ & $10500$ & $12500$ \\
  \midrule
  GBGS & IT & $ 554$ & $597$ & $738$ & $835$ & $ 1110$ \\
  & CPU & $ 71.5051$ & $163.6368$ & $265.7505$ & $403.0384$ & $1045.7000$ \\
  MRBGS& IT & $142$ & $153$ & $180$ & $198$ & $255$ \\
  & CPU & $140.5002$ & $ 358.1978$ & $608.0710$ & $1083.0000$ & $2409.0000$ \\
  FBCD & IT & $ 552$ & $613$ & $756$ & $852$ & $ 1103$ \\
  & CPU & $28.5252$ & $83.6194$ & $154.0855$ & $ 213.7058$ & $437.9306$ \\
  mADBCD & $\beta$ & $0.50$ & $0.55$ & $0.60$ & $0.50$ & $0.65$ \\
  & IT & $37$ & $37$ & $41$ & $44$ & $52$ \\
  & CPU & $1.6140$ & $2.6188$ & $5.0988$ & $7.5382$ & $12.7072$ \\
  speed-up{\_GBGS}& & $ 44.30$ & $62.49$ & $ 52.12$ & $53.47$ & $ 82.29$  \\
  speed-up{\_MRBGS}& & $87.05$ & $136.78$ & $119.26$ & $ 143.67$ & $189.58$  \\
  speed-up{\_FBCD}& & $17.67$ & $31.93$ & $ 30.22$ & $28.35$ & $ 34.46$  \\
  \bottomrule
  \end{tabular*}
\end{table}

From Tables \ref{tab:1}-\ref{tab:4}, it can be observed that for solving the 20 linear least-squares problems, all the block iterative GBGS, MRBGS, FBCD and mADBCD methods are convergent. When an appropriate momentum parameter $\beta$ is chosen, the numerical performance of the mADBCD method is optimal, as indicated by the iteration counts and CPU times. The convergence speed of the FBCD method is faster than that of GBGS and MRBGS methods. Additionally, when $m/n\leq 0.5$, the GBGS method outperforms the MRBGS method, see Tables \ref{tab:3} and \ref{tab:4}. Furthermore, the ranges of speed-up{\_GBGS},  speed-up{\_MRBGS} and speed-up{\_FBCD} are distributed between 8.91 and 82.29, 7.76 and 189.58, 4.55 and 34.46, respectively. These data further reflect that the best numerical performance among these four block coordinate descent type methods is the mADBCD method, followed by the FBCD method.
\par
In Figures \ref{fig:6}-\ref{fig:10}, we present the curves depicting the relative solution error (RSE) against the number of iterations (left) and CPU time (right) for the GBGS, MRBGS, FBCD and mADBCD methods. From these five figures, we can observe that the experimental phenomena are almost the same as those shown in Tables \ref{tab:1}-\ref{tab:4}, i.e., the mADBCD method has the least number of iterations and CPU time when the iterations satisfy the termination criterion \eqref{termination}.
\par
Figures \ref{fig:11}-\ref{fig:15} illustrate the significant impact of different values of the momentum parameter $\beta$ on the efficiency of solving sparse linear least-squares problems from Example \ref{example 5.2} using the mADBCD method. We have also realized that choosing a suitably appropriate parameter $\beta$ can significantly enhance the convergence speed of the mADBCD method and optimize its numerical performance.
\par
For solving Example \ref{example 5.2}, it can be concluded that the most competitive numerical performance is the mADBCD method, followed by the FBCD method, and the least effective is the GBGS method according to Tables \ref{tab:6}-\ref{tab:10}. From Tables \ref{tab:6}-\ref{tab:9}, it can be seen that although the numbers of iterations of the GBGS and MRBGS methods are less than that of the FBCD method, the FBCD method is more efficient in terms of CPU times, which is similarly reflected in the experimental phenomena in Tables \ref{tab:1}-\ref{tab:4}. Moreover, the ranges of speed-up{\_GBGS}, speed-up{\_MRBGS} and speed-up{\_FBCD} are distributed between 9.80 and 668.03, 5.80 and 546.07, 2.60 and 81.86, respectively, which further emphasizes the numerical competitiveness of the mADBCD method for solving linear least-squares problems.
\par
For the image reconstruction problems in Examples \ref{example 5.3} and \ref{example 5.4}, we run the GBGS, MRBGS, FBCD, and mADBCD methods for 180 seconds and present their numerical results in Figures \ref{fig:20} and \ref{fig:21}, respectively. From these two figures, we can conclude that, in the same CPU time, the efficiency of image reconstruction is best for the mADBCD method, followed by the FBCD method, and the worst is the GBGS method, which further indicates the numerical superiority of the mADBCD method.
\par
Whether the coefficient matrix of problem \eqref{linear least} is dense or sparse, the mADBCD method demonstrates the best numerical performance when an appropriate relaxation parameter is selected. In conclusion, the above numerical experiments indicate that this method is the fastest and most efficient compared to recent block coordinate descent methods for solving problem \eqref{linear least}.

\begin{figure}[!htbp]
  \renewcommand\figurename{Figure}
    \centering
  \begin{minipage}[t]{0.45\linewidth}
    \centering
    \includegraphics[scale=0.40]{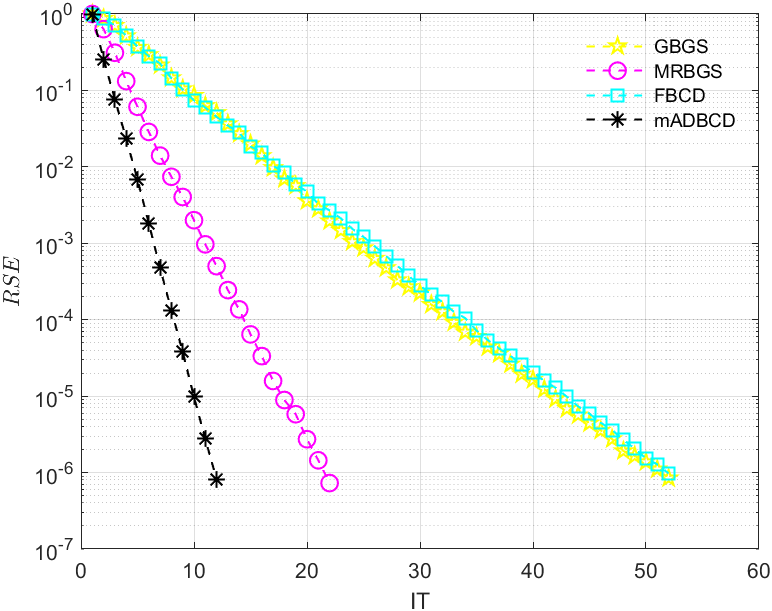}
  \end{minipage}
  \begin{minipage}[t]{0.45\linewidth}
    \centering
    \includegraphics[scale=0.40]{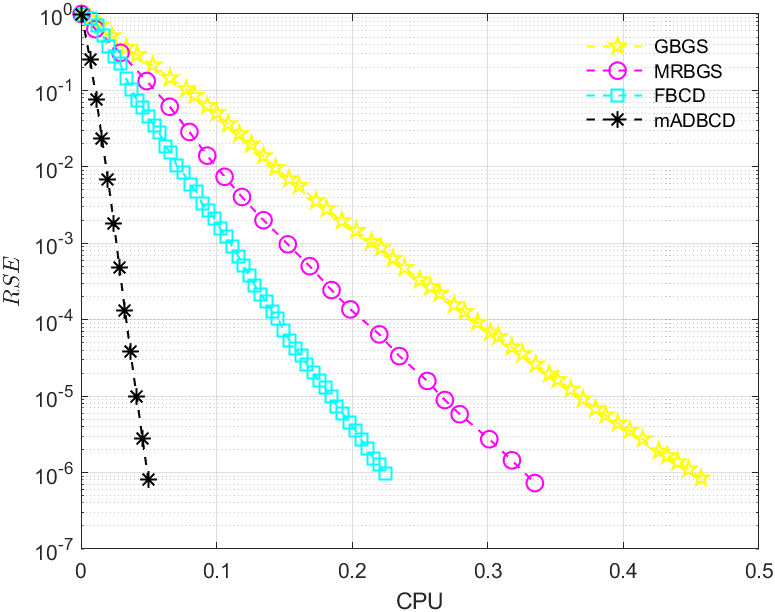}
  \end{minipage}
  \vspace{0.0cm}
  \caption{RSE versus IT (left) and CPU (right) of GBGS, MRBGS, FBCD and mADBCD for coefficient matrix $\mathbf{A}$=\textbf{randn}$(7500,750)$.}
  \label{fig:6}
\end{figure}
\begin{figure}[!htbp]
  \renewcommand\figurename{Figure}
    \centering
  \begin{minipage}[t]{0.45\linewidth}
    \centering
    \includegraphics[scale=0.40]{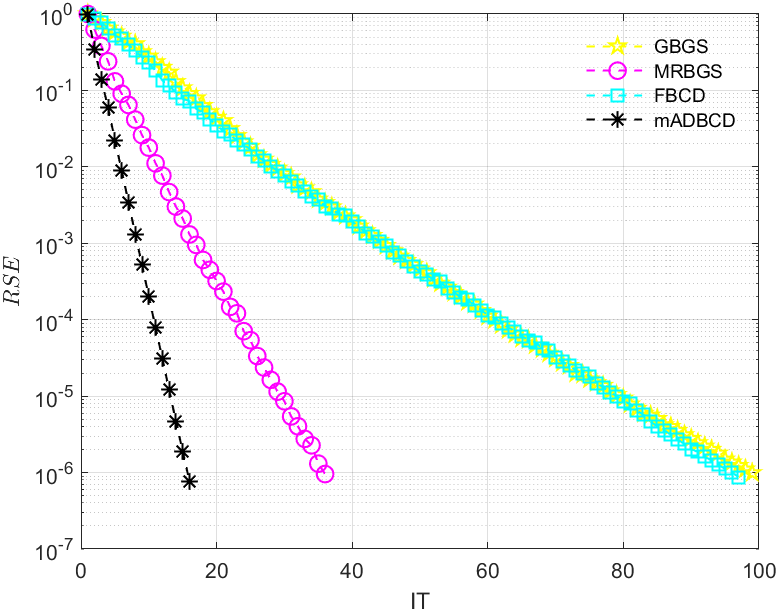}
  \end{minipage}
  \begin{minipage}[t]{0.45\linewidth}
    \centering
    \includegraphics[scale=0.40]{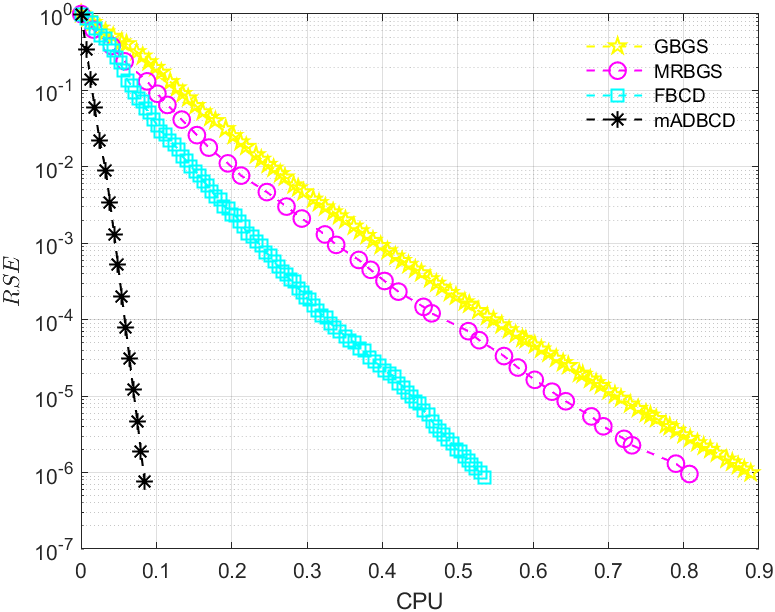}
  \end{minipage}
  \vspace{0.0cm}
  \caption{RSE versus IT (left) and CPU (right) of GBGS, MRBGS, FBCD and mADBCD for coefficient matrix $\mathbf{A}$=\textbf{randn}$(5500,1100)$.}
  \label{fig:7}
\end{figure}
\begin{figure}[!htbp]
  \renewcommand\figurename{Figure}
    \centering
  \begin{minipage}[t]{0.45\linewidth}
    \centering
    \includegraphics[scale=0.40]{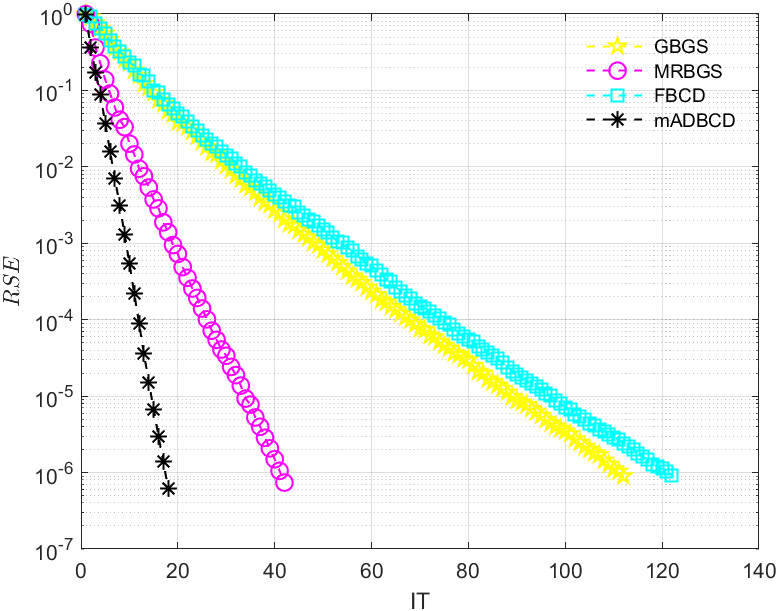}
  \end{minipage}
  \begin{minipage}[t]{0.45\linewidth}
    \centering
    \includegraphics[scale=0.40]{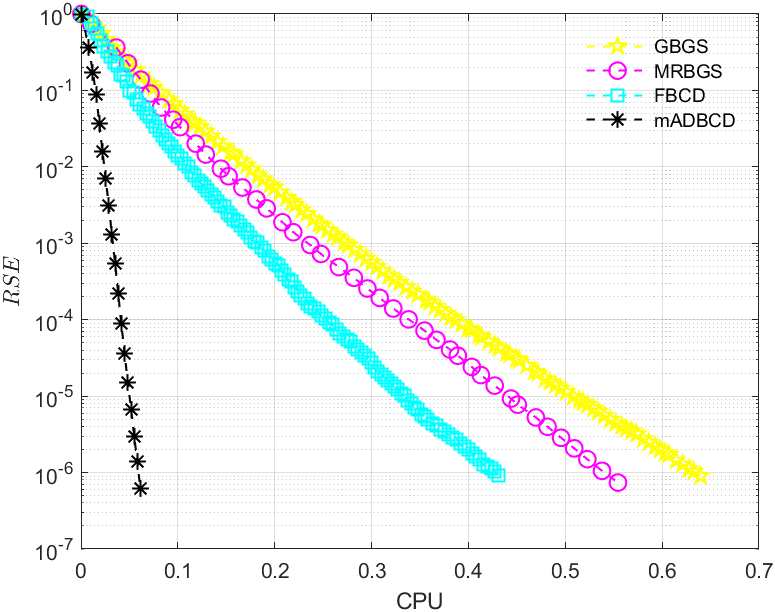}
  \end{minipage}
  \vspace{0.0cm}
  \caption{RSE versus IT (left) and CPU (right) of GBGS, MRBGS, FBCD and mADBCD for coefficient matrix $\mathbf{A}$=\textbf{randn}$(4000,1000)$.}
  \label{fig:8}
\end{figure}
\begin{figure}[!htbp]
  \renewcommand\figurename{Figure}
    \centering
  \begin{minipage}[t]{0.45\linewidth}
    \centering
    \includegraphics[scale=0.40]{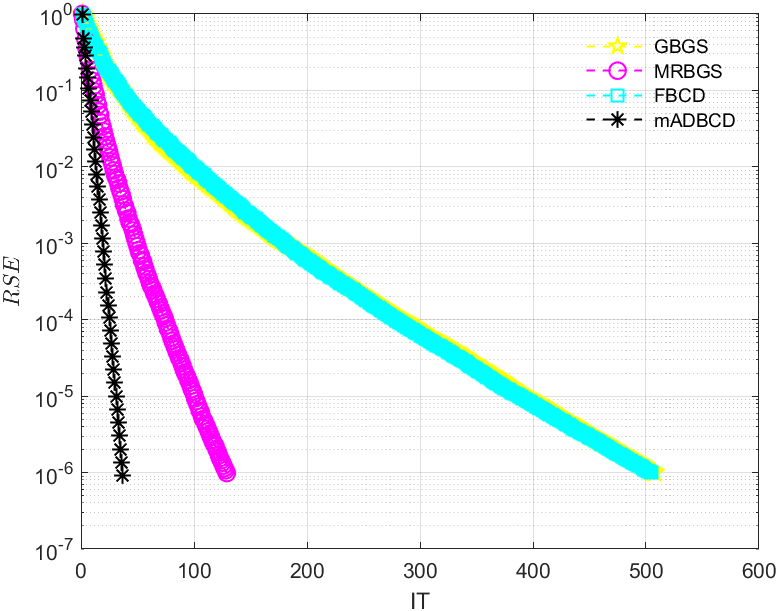}
  \end{minipage}
  \begin{minipage}[t]{0.45\linewidth}
    \centering
    \includegraphics[scale=0.40]{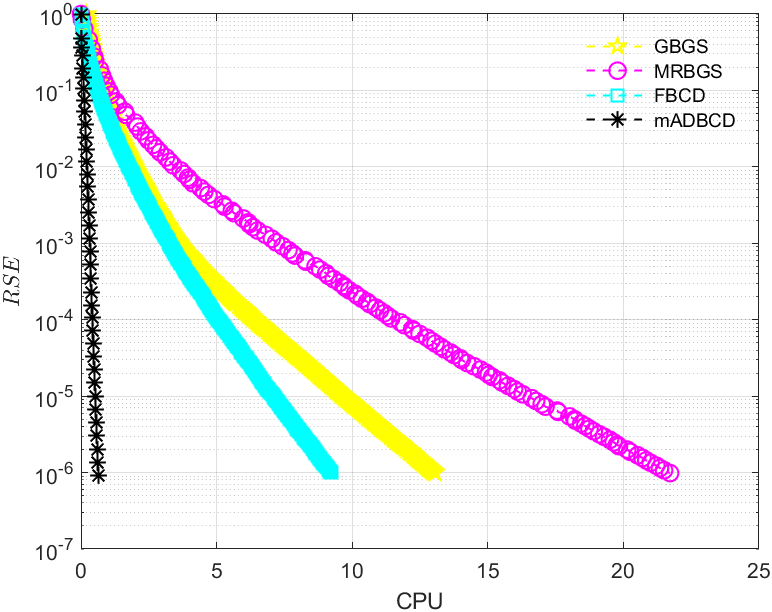}
  \end{minipage}
  \vspace{0.0cm}
  \caption{RSE versus IT (left) and CPU (right) of GBGS, MRBGS, FBCD and mADBCD for coefficient matrix $\mathbf{A}$=\textbf{randn}$(6000,3000)$.}
  \label{fig:9}
\end{figure}

\begin{figure}[!htbp]
  \renewcommand\figurename{Figure}
    \centering
  \begin{minipage}[t]{0.45\linewidth}
    \centering
    \includegraphics[scale=0.40]{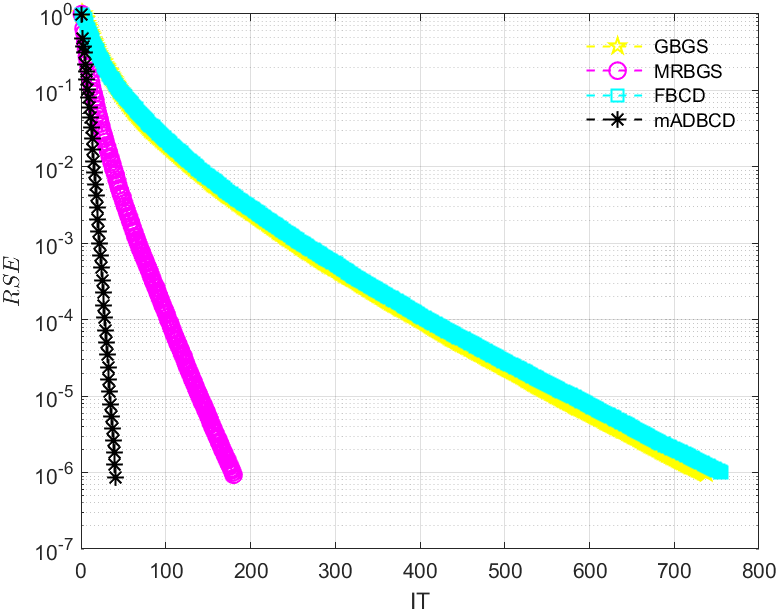}
  \end{minipage}
  \begin{minipage}[t]{0.45\linewidth}
    \centering
    \includegraphics[scale=0.40]{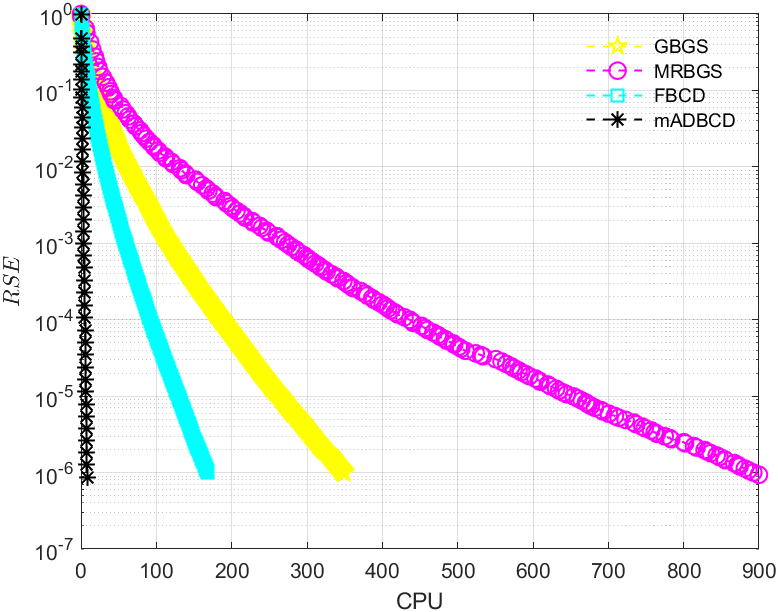}
  \end{minipage}
  \vspace{0.0cm}
  \caption{RSE versus IT (left) and CPU (right) of GBGS, MRBGS, FBCD and mADBCD for coefficient matrix $\mathbf{A}$=\textbf{randn}$(16000,8500)$.}
  \label{fig:10}
\end{figure}
\begin{table}[!htbp]
  \centering
  \caption{Information of the matrices from Florida space matrix collection.}\label{tab:5}%
  \begin{tabular*}{\hsize}{@{}@{\extracolsep{\fill}}lllllll@{}}
  \toprule
  name & & ash958 & abtaha1 & abtaha2 &  ash608 & WorldCities \\
  $m \times n$  & & $958 \times 292$ & $14596 \times 209$ & $37932 \times 331$ & $608 \times 188$ & $315 \times 100$ \\
  density & & $0.68 \%$ & $1.68 \%$ & $1.09\%$ & $1.06 \%$ & $23.87\%$ \\
  cond$(\mathbf{A})$ & & $3.20$ & $12.23$ & $12.22$ & $3.37$ & $66.00$ \\
  \midrule
  name & & well1033 & well1850 & rail516 &  rail582 & r05 \\
  $m \times n$ & & $1033 \times 320$ & $1850 \times 712$ & $516 \times 47827$ & $582 \times 56097$ & $5190 \times 9690 $ \\
  density & & $1.43 \%$ & $0.66 \%$ & $1.28\%$ & $1.28 \%$ & $1.23\%$ \\
  cond$(\mathbf{A})$ & & $166.13$ & $11.13$ & $143.89$ & $185.91$ & $121.82$ \\
  \midrule
  name & & cage8 & cage9 & cage10 & nemsafm &lp22 \\
  $m \times n$ & & $1015 \times 1015$ & $5226 \times 14721$ & $11397 \times 11397$ & $334 \times 2348$ & $2958 \times 16392$ \\
  density & & $1.07\%$ & $0.33 \%$ & $0.16\%$ & $0.36\%$ & $0.14\%$ \\
  cond$(\mathbf{A})$ &  & $11.41$ & $12.60$ & $11.02$  & $4.77$ & $25.78$ \\
  \midrule
  name& & model1  & model8 & nemscem & P05 & pgp2  \\
  $m \times n$ & & $362 \times 798$ & $2896 \times 6464$ & $651 \times 1712$ & $5090 \times 9590$ & $4034\times 13254$ \\
  density & & $1.05 \%$ & $0.14 \%$ & $ 0.31\%$ & $0.12\%$ & $0.14\%$ \\
  cond$(\mathbf{A})$ & & $17.57$  & $53.63$ & $44.63$ & $85.37$ & $31.95$ \\
  \bottomrule
  \end{tabular*}
\end{table}

\begin{figure}[!htbp]
  \renewcommand\figurename{Figure}
    \centering
  \begin{minipage}[t]{0.45\linewidth}
    \centering
    \includegraphics[scale=0.40]{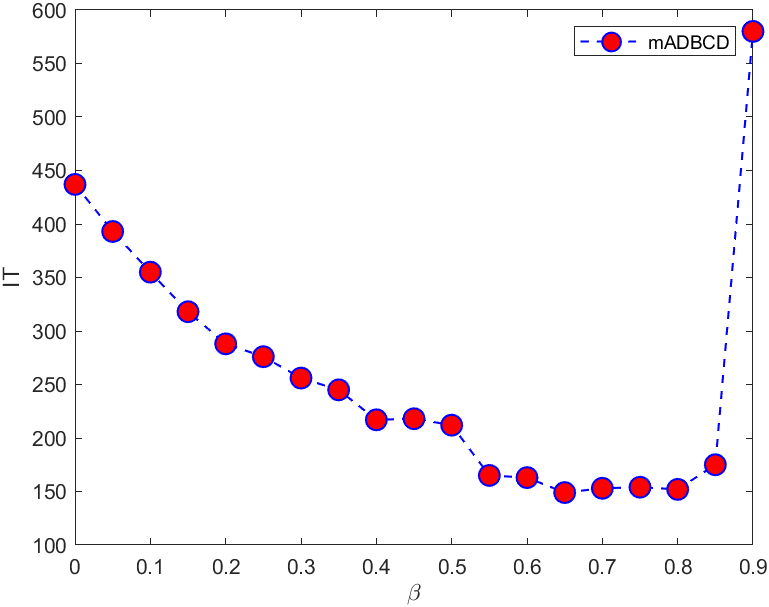}
  \end{minipage}
  \begin{minipage}[t]{0.45\linewidth}
    \centering
    \includegraphics[scale=0.40]{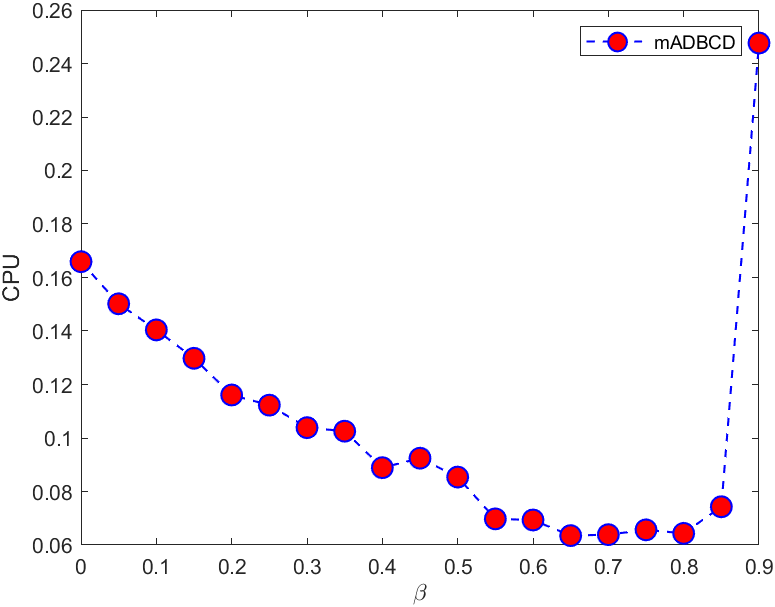}
  \end{minipage}
  \vspace{0.0cm}
  \caption{Pictures of $\beta$ versus IT (left) and CPU (right) for mADBCD when $\mathbf{A}$ is abtaha2.}
  \label{fig:11}
\end{figure}
\begin{figure}[!htbp]
  \renewcommand\figurename{Figure}
    \centering
  \begin{minipage}[t]{0.45\linewidth}
    \centering
    \includegraphics[scale=0.40]{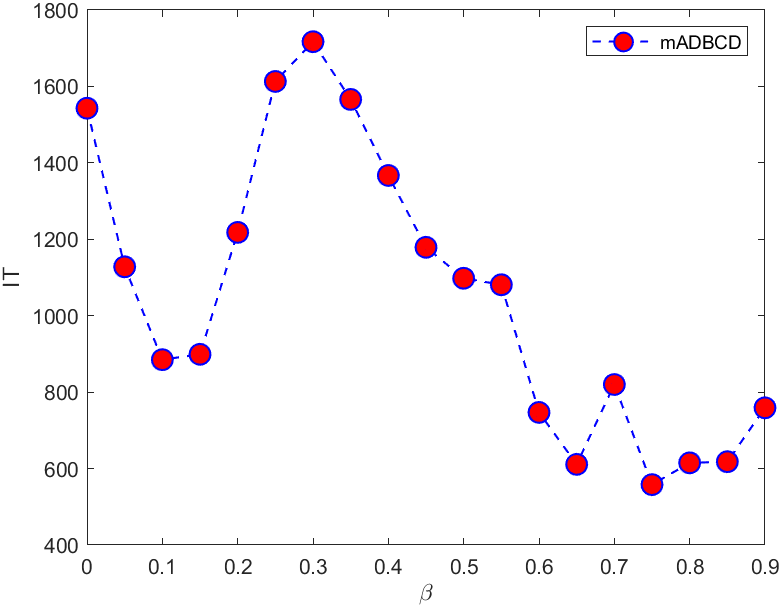}
  \end{minipage}
  \begin{minipage}[t]{0.45\linewidth}
    \centering
    \includegraphics[scale=0.40]{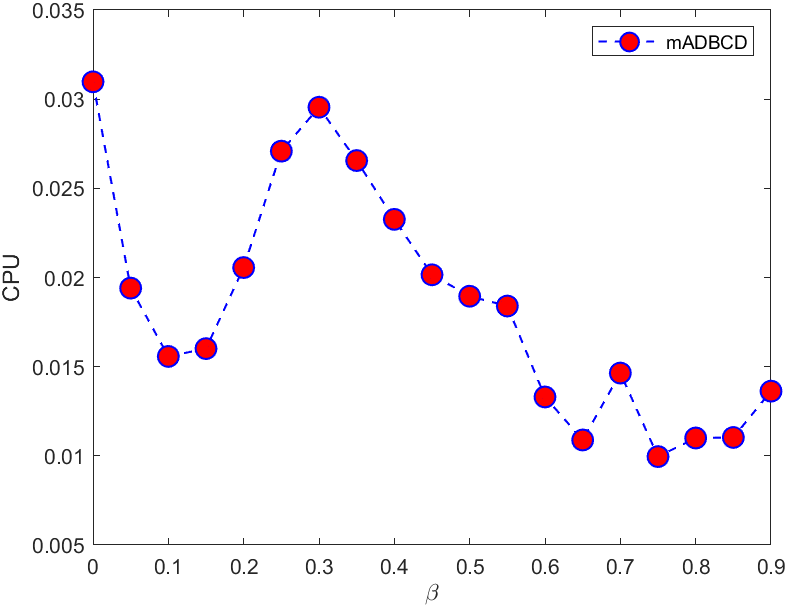}
  \end{minipage}
  \vspace{0.0cm}
  \caption{Pictures of $\beta$ versus IT (left) and CPU (right) for mADBCD when $\mathbf{A}$ is WorldCities.}
  \label{fig:12}
\end{figure}
\begin{figure}[!htbp]
  \renewcommand\figurename{Figure}
    \centering
  \begin{minipage}[t]{0.45\linewidth}
    \centering
    \includegraphics[scale=0.40]{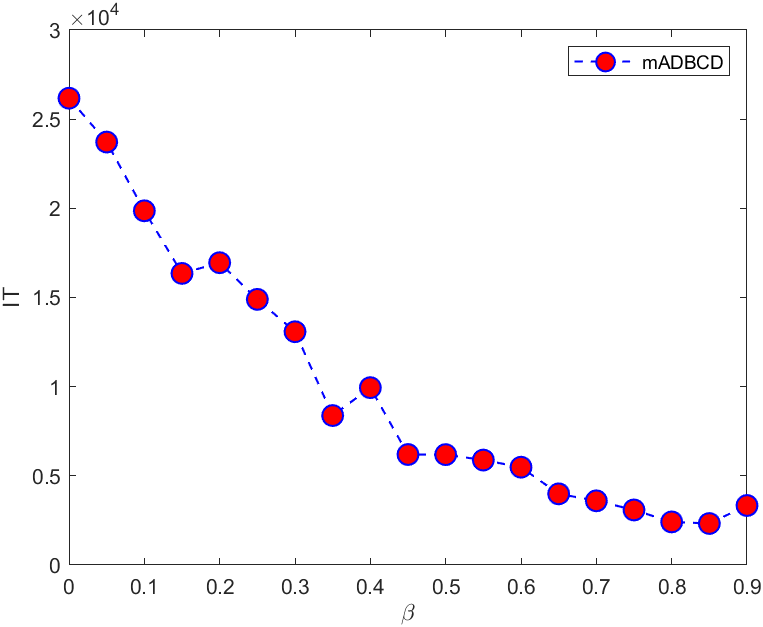}
  \end{minipage}
  \begin{minipage}[t]{0.45\linewidth}
    \centering
    \includegraphics[scale=0.40]{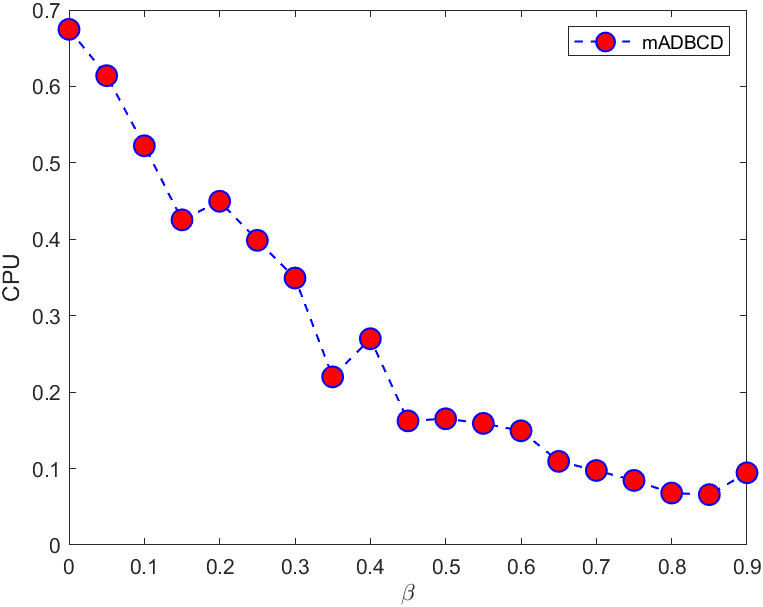}
  \end{minipage}
  \vspace{0.0cm}
  \caption{Pictures of $\beta$ versus IT (left) and CPU (right) for mADBCD when $\mathbf{A}$ is well1850.}
  \label{fig:13}
\end{figure}
\begin{figure}[!htbp]
  \renewcommand\figurename{Figure}
    \centering
  \begin{minipage}[t]{0.45\linewidth}
    \centering
    \includegraphics[scale=0.40]{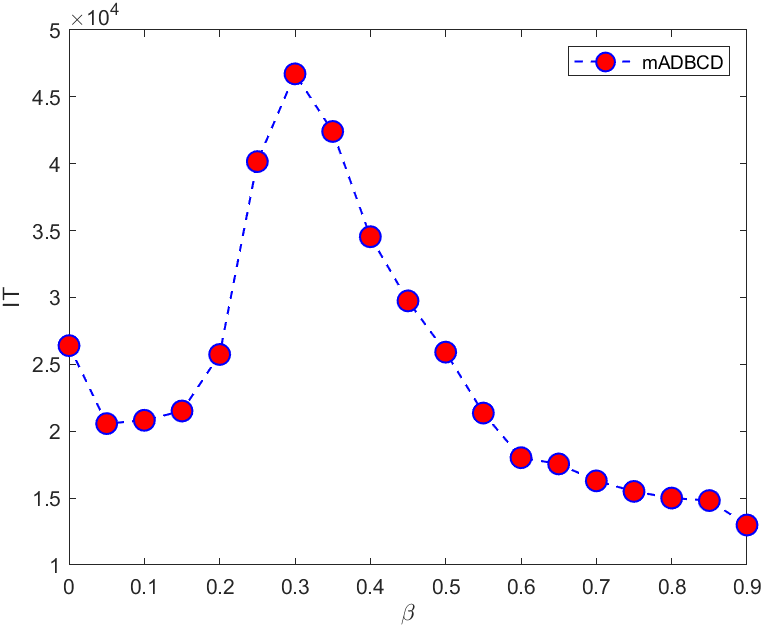}
  \end{minipage}
  \begin{minipage}[t]{0.45\linewidth}
    \centering
    \includegraphics[scale=0.40]{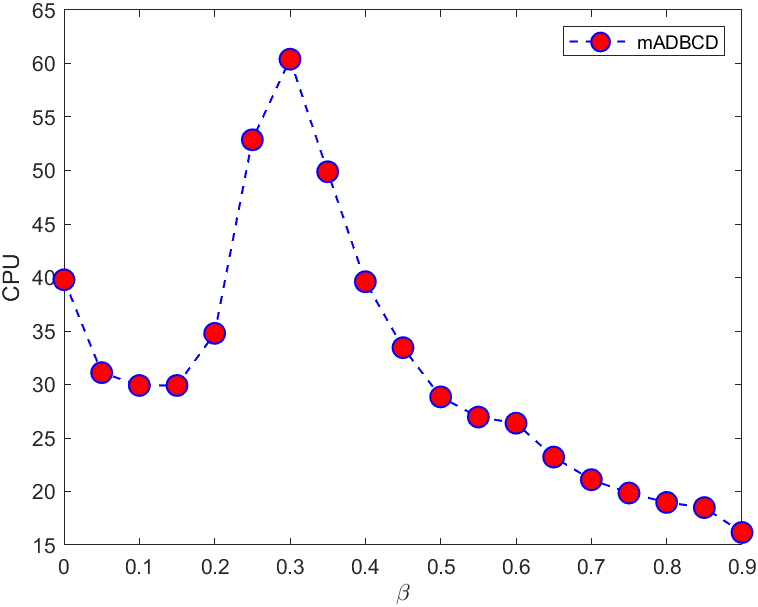}
  \end{minipage}
  \vspace{0.0cm}
  \caption{Pictures of $\beta$ versus IT (left) and CPU (right) for mADBCD when $\mathbf{A}$ is rail516$^T$.}
  \label{fig:14}
\end{figure}
\begin{figure}[!htbp]
  \renewcommand\figurename{Figure}
    \centering
  \begin{minipage}[t]{0.45\linewidth}
    \centering
    \includegraphics[scale=0.40]{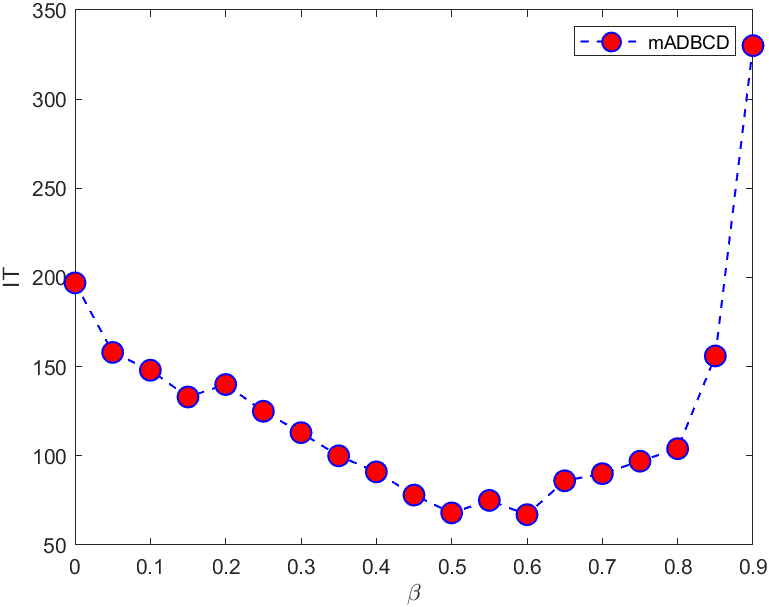}
  \end{minipage}
  \begin{minipage}[t]{0.45\linewidth}
    \centering
    \includegraphics[scale=0.40]{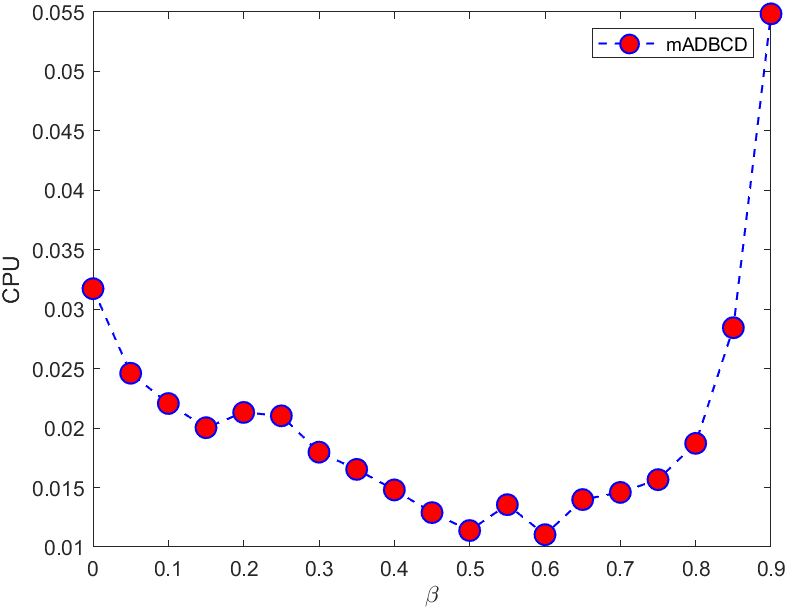}
  \end{minipage}
  \vspace{0.0cm}
  \caption{Pictures of $\beta$ versus IT (left) and CPU (right) for mADBCD when $\mathbf{A}$ is cage9.}
  \label{fig:15}
\end{figure}
\begin{table}[!htbp]
  \centering
  \caption{Numerical results of GBGS, MRBGS, FBCD and mADBCD for Florida sparse matrix collection.}\label{tab:6}%
  \begin{tabular*}{\hsize}{@{}@{\extracolsep{\fill}}lllllll@{}}
  \toprule
  name & & ash958 & abtaha1 & abtaha2 &  ash608 & WorldCities \\
  \midrule
  GBGS & IT & $53$ & $771$ & $1277$ & $47$ & $ 1448$ \\
  & CPU & $0.0078$ & $ 1.6182$ & $5.8035$ & $0.0049$ & $0.4238$ \\
  MRBGS& IT & $24$ & $345$ & $498$ & $23$ & $1338$ \\
  & CPU & $ 0.0060$ & $ 1.3409$ & $5.6426$ & $0.0029$ & $0.4141$ \\
  FBCD & IT & $69$ & $450$ & $1245$ & $63$ & $ 2078$ \\
  & CPU & $0.0022$ & $ 0.0699$ & $0.9450$ & $ 0.0013$ & $ 0.0536$ \\
  mADBCD & $\beta$ & $0.30$ & $0.35$ & $0.65$ & $0.30$ & $0.75$ \\
  & IT & $17$ & $116$ & $149$ & $19$ & $558$ \\
  & CPU & $0.0005$ & $0.0196$ & $0.0540$ & $ 0.0005$ & $0.0127$ \\
  speed-up{\_GBGS}& & $15.60$ & $82.56$ & $107.47$ & $9.80$ & $33.37$  \\
  speed-up{\_MRBGS}& & $12.00$ & $68.41$ & $104.49$ & $5.80$ & $32.61$  \\
  speed-up{\_FBCD}& & $ 4.40$ & $3.57$ & $17.50$ & $ 2.60$ & $4.22$  \\
  \bottomrule
  \end{tabular*}  
\end{table}

\begin{table}[!htbp]
  \centering
  \caption{Numerical results of GBGS, MRBGS, FBCD and mADBCD for Florida sparse matrix collection.}\label{tab:7}%
  \begin{tabular*}{\hsize}{@{}@{\extracolsep{\fill}}lllllll@{}}
  \toprule
  name & & well1033 & well1850 & rail516$^T$ &  rail582$^T$ &r05$^T$ \\
  \midrule
  GBGS & IT & $73057$ & $113952$ & $62672$ & $101317$ & $ 19922$ \\
  & CPU & $8.2718$ & $32.5437$ & $434.3091$ & $716.4277$ & $39.4501$ \\
  MRBGS& IT & $50625$ & $22714$ & $27246$ & $96403$ & $ 4921$ \\
  & CPU & $5.9303$ & $12.2070$ & $362.0436$ & $1164.4000 $ & $17.6058$ \\
  FBCD & IT & $103731$ & $142306$ & $94850$ & $165856$ & $ 52648$ \\
  & CPU & $ 1.7977$ & $4.2263$ & $124.6174$ & $219.4461$ & $ 21.1645$ \\
  mADBCD & $\beta$ & $0.90$ & $0.85$ & $0.90$ & $0.90$ & $0.90$ \\
  & IT & $ 6927$ & $2334$ & $12990$ & $19866$ & $2387$ \\
  & CPU & $0.1039$ & $ 0.0624$ & $13.5245$ & $24.7864$ & $0.7799$ \\
  speed-up{\_GBGS}& & $79.61$ & $521.53$ & $32.11$ & $28.90$ & $50.58$  \\
  speed-up{\_MRBGS}& & $57.08$ & $195.63$ & $26.77$ & $46.98$ & $22.57$  \\
  speed-up{\_FBCD}& & $17.30$ & $67.73$ & $9.21$ & $8.85$ & $ 27.14$  \\
  \bottomrule
  \end{tabular*}  
\end{table}

\begin{table}[!htbp]
  \centering
  \caption{Numerical results of GBGS, MRBGS, FBCD and mADBCD for Florida sparse matrix collection.}\label{tab:8}%
  \begin{tabular*}{\hsize}{@{}@{\extracolsep{\fill}}lllllll@{}}
  \toprule
  name & & cage8 & cage9 & cage10 & nemsafm$^T$ &lp22$^T$ \\
  \midrule
  GBGS & IT & $109$ & $227$ & $145$ & $73$ & $ 19922$ \\
  & CPU & $ 0.0445$ & $ 0.3244$ & $1.7932$ & $0.0191$ & $39.4501$ \\
  MRBGS& IT & $61$ & $97$ & $53$ & $41$ & $ 4921$ \\
  & CPU & $0.0496$ & $  0.2945$ & $ 1.7576$ & $ 0.0086$ & $17.6058$ \\
  FBCD & IT & $280$ & $425$ & $250$ & $132$ & $ 52648$ \\
  & CPU & $ 0.0101$ & $ 0.0550$ & $ 0.1599$ & $0.0026$ & $ 21.1645$ \\
  mADBCD & $\beta$ & $0.55$ & $0.60$ & $0.45$ & $0.45$ & $0.05$ \\
  & IT & $83$ & $67$ & $52$ & $31$ & $6028$ \\
  & CPU & $ 0.0026$ & $0.0088$ & $0.0238$ & $ 0.0005$ & $1.2944$ \\
  speed-up{\_GBGS}& & $17.11$ & $36.86$ & $75.34$ & $ 38.20$ & $47.39$  \\
  speed-up{\_MRBGS}& & $19.08$ & $33.47$ & $73.85$ & $17.20$ & $45.54$  \\
  speed-up{\_FBCD}& & $ 3.88$ & $6.25$ & $6.72$ & $5.20$ & $ 11.21$  \\
  \bottomrule
  \end{tabular*}  
\end{table}

\begin{table}[!htbp]
  \centering
  \caption{Numerical results of GBGS, MRBGS, FBCD and mADBCD for Florida sparse matrix collection.}\label{tab:9}%
  \begin{tabular*}{\hsize}{@{}@{\extracolsep{\fill}}lllllll@{}}
  \toprule
  name & & model1$^T$ & model8$^T$ & nemscem$^T$ & p05$^T$ &pgp2$^T$ \\
  \midrule
  GBGS & IT & $485$ & $49551$ & $1090$ & $ 16581$ & $  1894$ \\
  & CPU & $0.0701$ & $41.7820$ & $0.2240$ & $ 34.9072$ & $36.2071$ \\
  MRBGS& IT & $177$ & $8214$ & $519$ & $4685$ & $989$ \\
  & CPU & $0.0417$ & $12.7197$ & $0.1398$ & $21.3427$ & $ 29.5970$ \\
  FBCD & IT & $776$ & $60880$ & $2705$ & $48300$ & $ 2512$ \\
  & CPU & $ 0.0195$ & $5.9184$ & $0.0513$ & $ 18.4968$ & $ 0.5408$ \\
  mADBCD & $\beta$ & $0.80$ & $0.90$ & $0.80$ & $0.90$ & $0.80$ \\
  & IT & $118$ & $832$ & $620$ & $1354$ & $388$ \\
  & CPU & $0.0013$ & $0.0723$ & $0.0089$ & $0.3384$ & $0.0542$ \\
  speed-up{\_GBGS}& & $53.92$ & $577.90$ & $25.17$ & $103.15$ & $668.03  $  \\
  speed-up{\_MRBGS}& & $32.08$ & $175.93$ & $15.71$ & $63.07$ & $546.07$  \\
  speed-up{\_FBCD}& & $15.00$ & $81.86$ & $5.76$ & $54.66$ & $9.98$  \\
  \bottomrule
  \end{tabular*}  
\end{table}
\begin{figure}[!htbp]
  \renewcommand\figurename{Figure}
    \centering
  \begin{minipage}[t]{0.45\linewidth}
    \centering
    \includegraphics[scale=0.40]{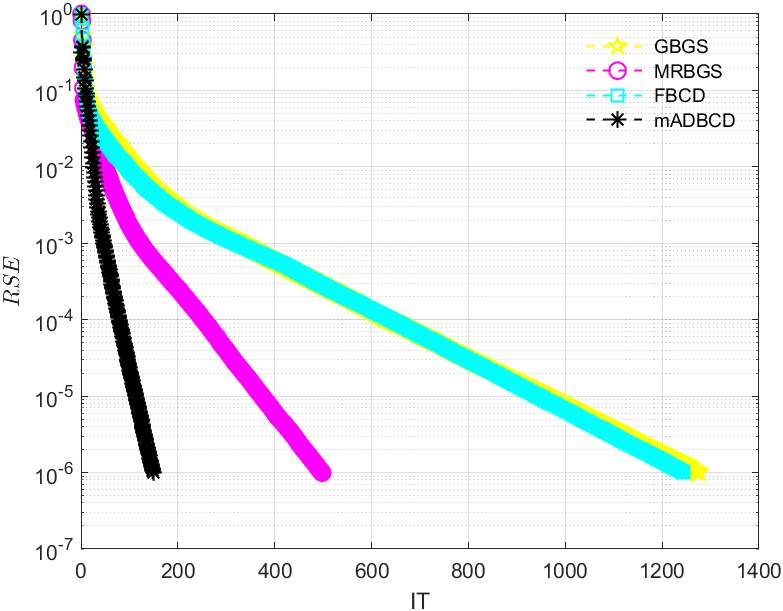}
  \end{minipage}
  \begin{minipage}[t]{0.45\linewidth}
    \centering
    \includegraphics[scale=0.40]{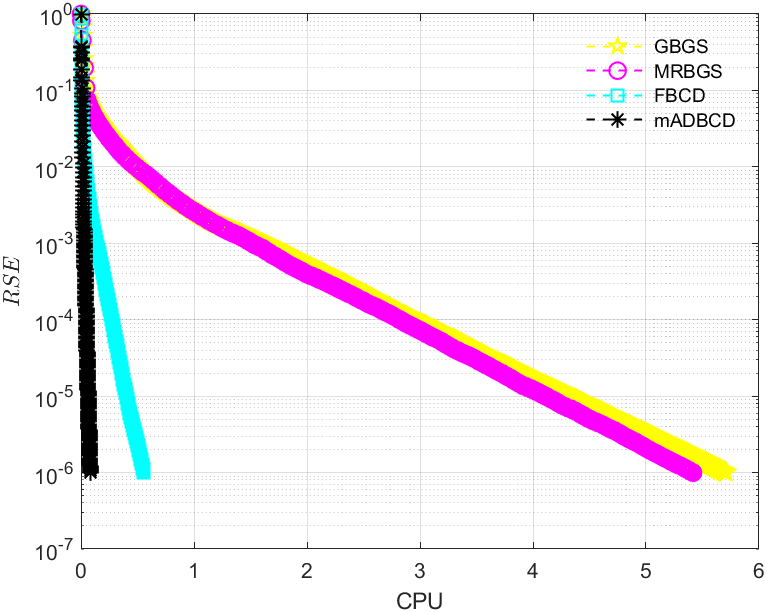}
  \end{minipage}
  \vspace{0.0cm}
  \caption{RSE versus IT (left) and CPU (right) of GBGS, MRBGS, FBCD and mADBCD for coefficient matrix abtaha2.}
  \label{fig:16}
\end{figure}
\begin{figure}[!htbp]
  \renewcommand\figurename{Figure}
    \centering
  \begin{minipage}[t]{0.45\linewidth}
    \centering
    \includegraphics[scale=0.40]{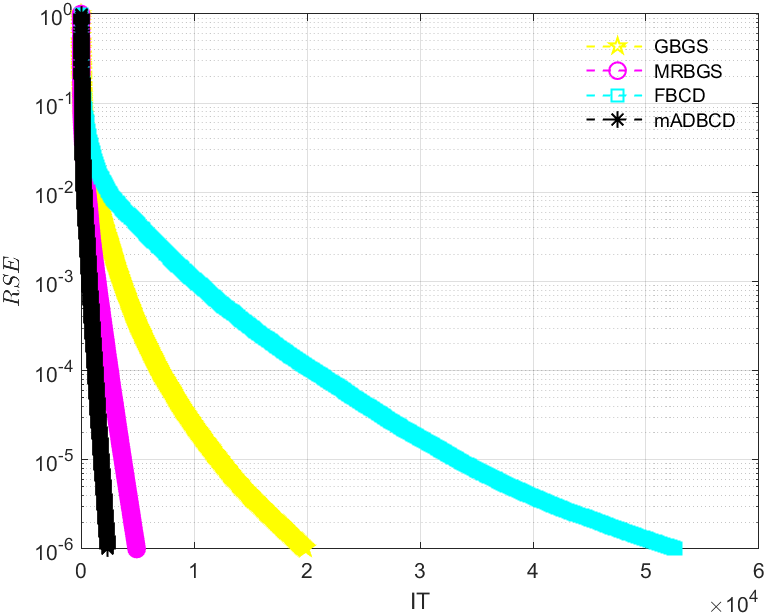}
  \end{minipage}
  \begin{minipage}[t]{0.45\linewidth}
    \centering
    \includegraphics[scale=0.40]{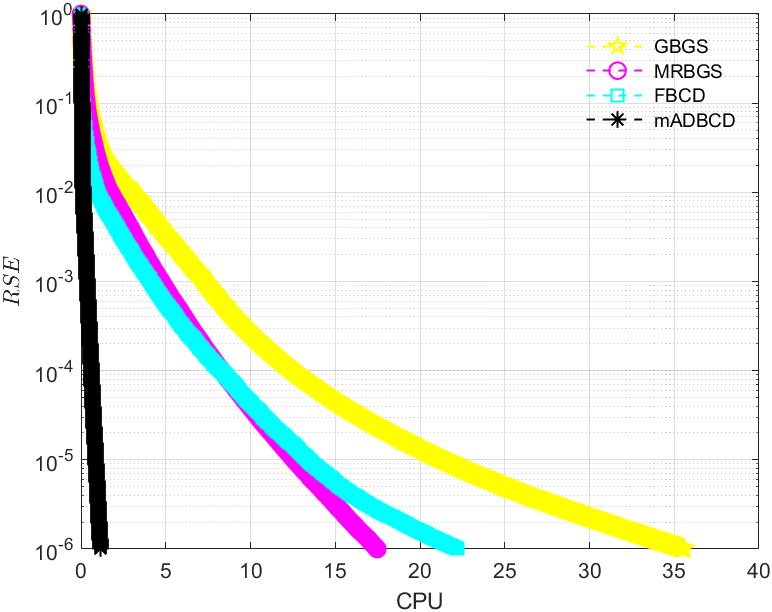}
  \end{minipage}
  \vspace{0.0cm}
  \caption{RSE versus IT (left) and CPU (right) of GBGS, MRBGS, FBCD and mADBCD for coefficient matrix r05$^T$.}
  \label{fig:17}
\end{figure}
\begin{figure}[!htbp]
  \renewcommand\figurename{Figure}
    \centering
  \begin{minipage}[t]{0.45\linewidth}
    \centering
    \includegraphics[scale=0.40]{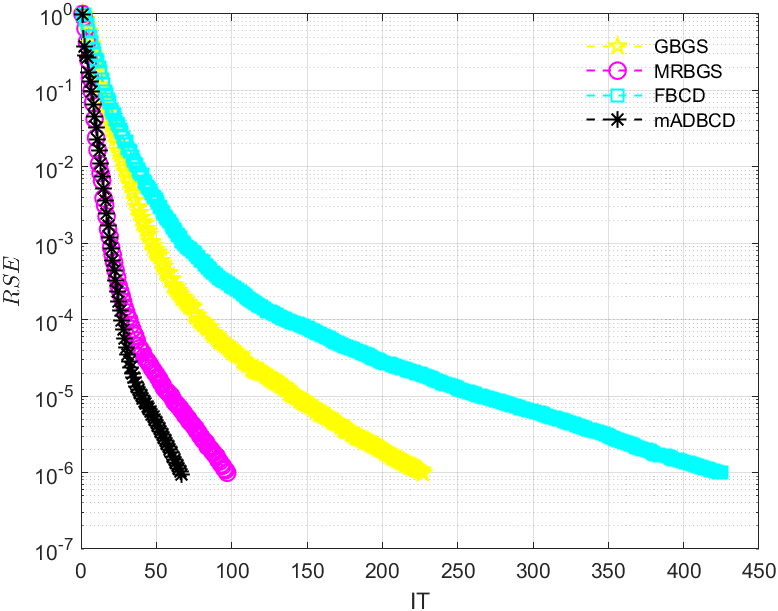}
  \end{minipage}
  \begin{minipage}[t]{0.45\linewidth}
    \centering
    \includegraphics[scale=0.40]{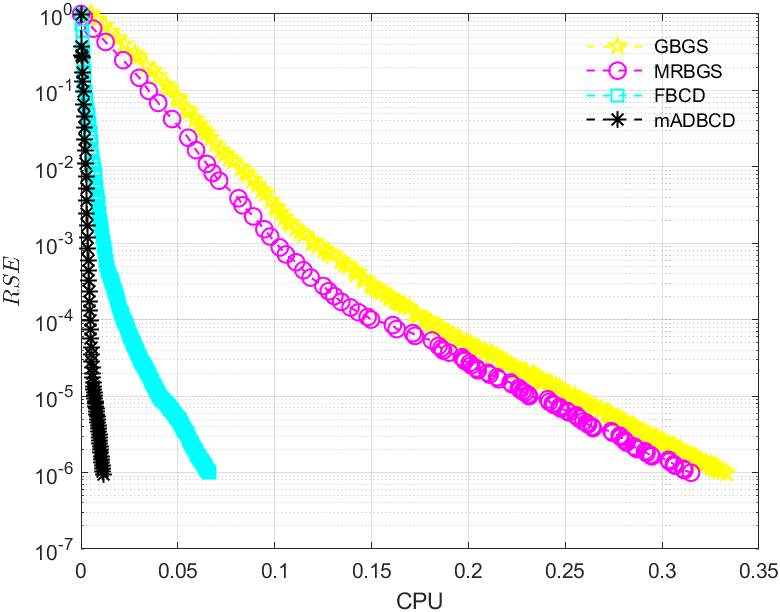}
  \end{minipage}
  \vspace{0.0cm}
  \caption{RSE versus IT (left) and CPU (right) of GBGS, MRBGS, FBCD and mADBCD for coefficient matrix cage9.}
  \label{fig:18}
\end{figure}
\begin{figure}[!htbp]
  \renewcommand\figurename{Figure}
    \centering
  \begin{minipage}[t]{0.45\linewidth}
    \centering
    \includegraphics[scale=0.40]{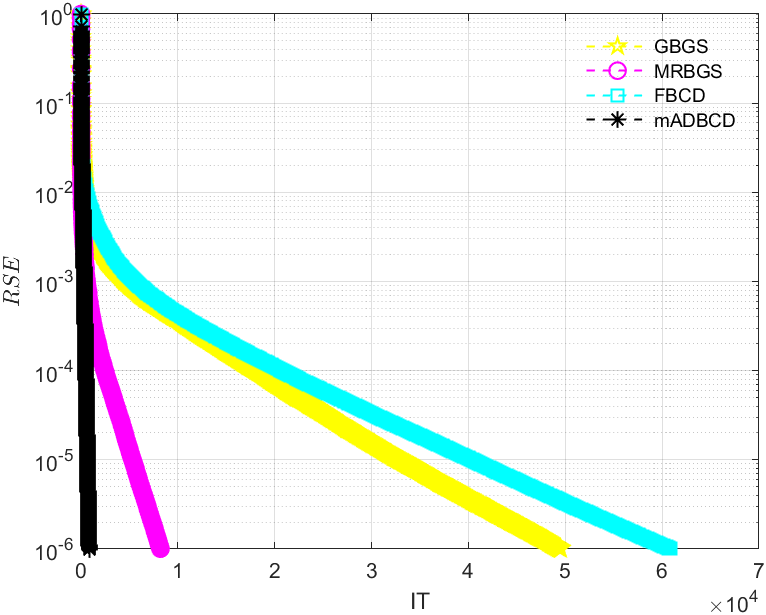}
  \end{minipage}
  \begin{minipage}[t]{0.45\linewidth}
    \centering
    \includegraphics[scale=0.40]{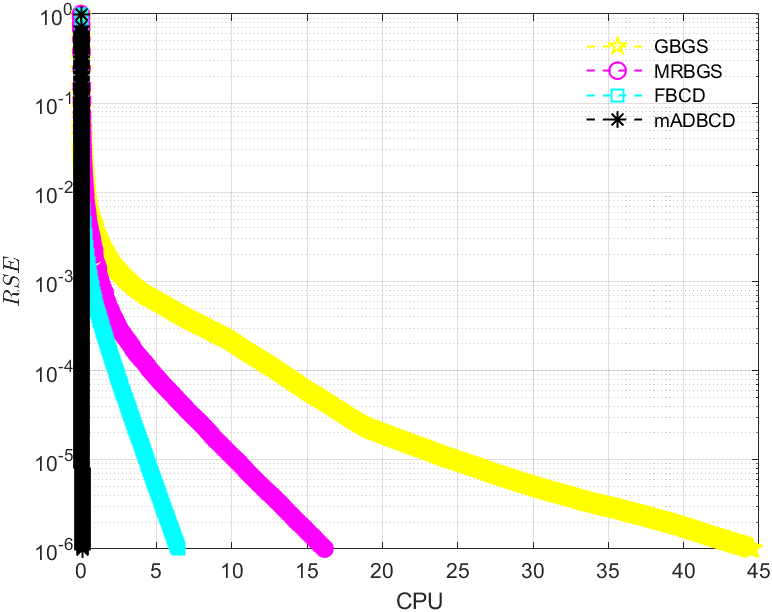}
  \end{minipage}
  \vspace{0.0cm}
  \caption{RSE versus IT (left) and CPU (right) of GBGS, MRBGS, FBCD and mADBCD for coefficient matrix model8$^T$.}
  \label{fig:19}
\end{figure}
\begin{figure}[!htbp]
  \renewcommand\figurename{Figure}
    \centering
  \begin{minipage}[t]{0.30\linewidth}
    \centering
    \includegraphics[scale=0.55]{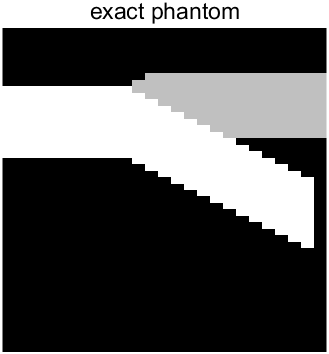}
  \end{minipage}
  \vspace{5mm}
  \begin{minipage}[t]{0.30\linewidth}
    \centering
    \includegraphics[scale=0.55]{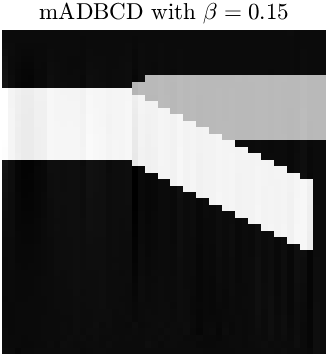}
  \end{minipage}
  \begin{minipage}[t]{0.30\linewidth}
    \centering
    \includegraphics[scale=0.55]{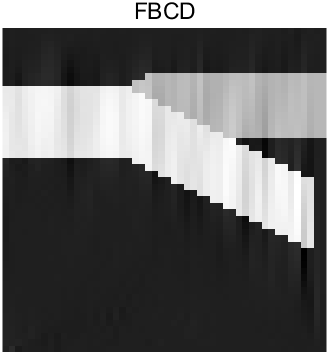}
  \end{minipage}
  \begin{minipage}[t]{0.30\linewidth}
    \centering
    \includegraphics[scale=0.55]{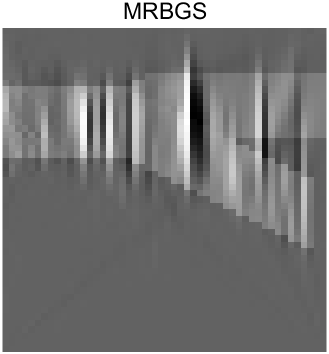}
  \end{minipage}
  \begin{minipage}[t]{0.30\linewidth}
    \centering
    \includegraphics[scale=0.55]{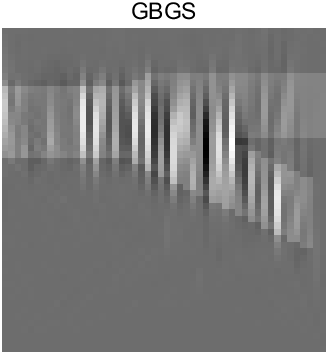}
  \end{minipage}
  \begin{minipage}[t]{0.30\linewidth}
    \includegraphics[scale=0.45]{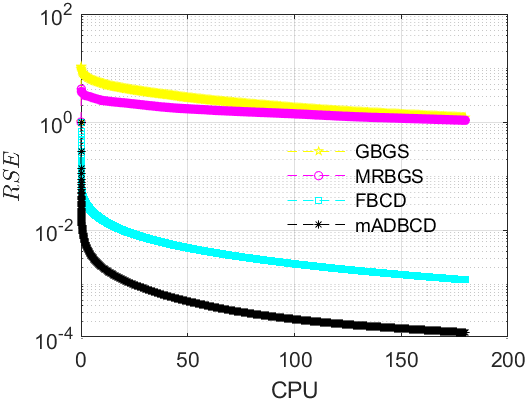}
  \end{minipage}
  \vspace{0.0cm}
  \caption{Performance of GBGS, MRBGS, FBCD and mADBCD methods for \textbf{seismictomo}$(n=50, s=80, p=120)$ test problem.}
  \label{fig:20}
\end{figure}
\begin{figure}[!htbp]
  \renewcommand\figurename{Figure}
    \centering
  \begin{minipage}[t]{0.30\linewidth}
    \centering
    \includegraphics[scale=0.55]{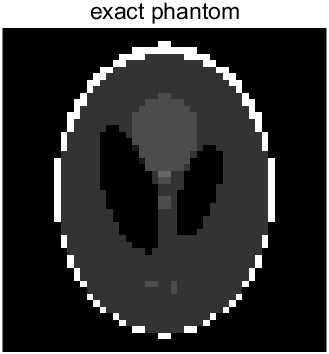}
  \end{minipage}
  \vspace{5mm}
  \begin{minipage}[t]{0.30\linewidth}
    \centering
    \includegraphics[scale=0.55]{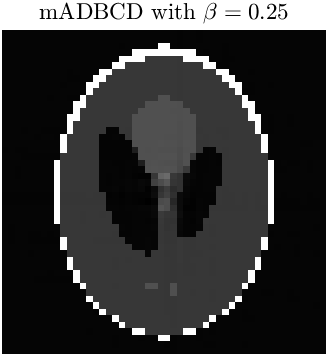}
  \end{minipage}
  \begin{minipage}[t]{0.30\linewidth}
    \centering
    \includegraphics[scale=0.55]{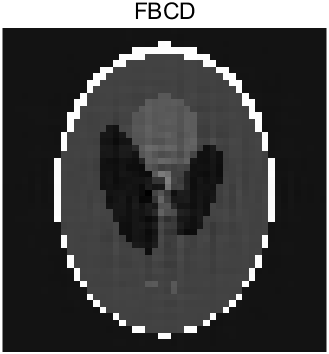}
  \end{minipage}
  \begin{minipage}[t]{0.30\linewidth}
    \centering
    \includegraphics[scale=0.55]{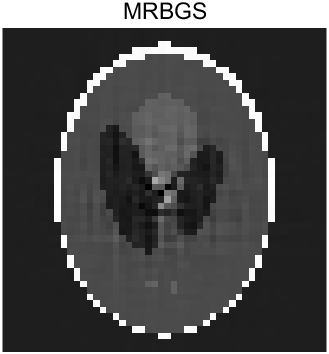}
  \end{minipage}
  \begin{minipage}[t]{0.30\linewidth}
    \centering
    \includegraphics[scale=0.55]{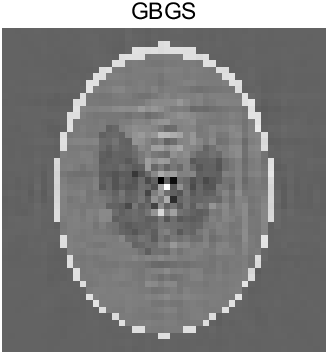}
  \end{minipage}
  \begin{minipage}[t]{0.30\linewidth}
    \includegraphics[scale=0.45]{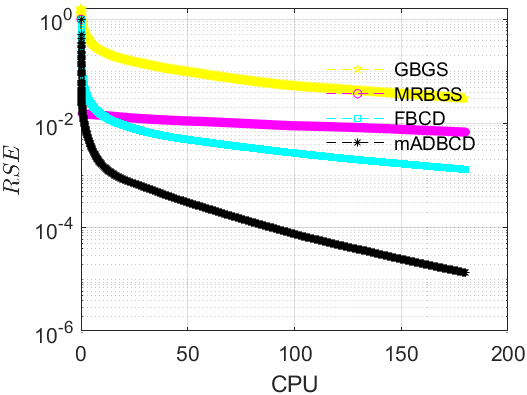}
  \end{minipage}
  \vspace{0.0cm}
  \caption{Performance of GBGS, MRBGS, FBCD and mADBCD methods for \textbf{fanlineartomo}$(N=50, \theta=0: 1: 300^{\circ}, P=50)$ test problem.}
  \label{fig:21}
\end{figure}
\subsection{Comparison of the CS-mADBCD method with the mADBCD method}\label{sec:5.2}
In the subsection, we compare the numerical performance of the mADBCD method with the CS-mADBCD method to solve extremely overdetermined linear least-squares problems from Example \ref{example 5.1} and Example \ref{example 5.5}. We also report the speed-up of the CS-mADBCD method against the mADBCD method, which is defined as
$$
\text { speed-up }=\frac{\mathrm{CPU} \text { of mADBCD }}{\text { CPU of CS-mADBCD }}.
$$
To reflect the fairness of the comparison, the values of the momentum parameters $\beta$ and $\beta_{CS}$ in Tables \ref{tab:10}-\ref{tab:13} are taken such that the mADBCD and CS-mADBCD methods each require the fewest number of iterations when they satisfy the termination criterion \eqref{termination}. In Tables \ref{tab:10}-\ref{tab:13}, CPU$_1$ refers to the time taken for generating the count sketch matrix $\mathbf{S}$ and computing  $\tilde{\mathbf{A}}=\mathbf{S}\mathbf{A}$ and $\tilde{\mathbf{b}}=\mathbf{S}\mathbf{b}$, which represents the time consumed for preprocessing the original least-squares problem \eqref{linear least} into
\begin{equation}\label{equ:5.2}
  \underset{\mathbf{x}\in\mathbb{R}^n}{\min}\|\tilde{\mathbf{b}}-\tilde{\mathbf{A}}\mathbf{x}\|_2.
\end{equation}
On the other hand, CPU$_2$ signifies the time taken to apply the mADBCD method to solve problem \eqref{equ:5.2} and satisfy the termination condition \eqref{termination} and CPU=CPU$_1$ + CPU$_2$. Although the convergence condition in Theorem \ref{theo:4.1} requires $d=O(n^2)$, in numerical experiments, to reduce the time needed for forming the least-squares problem \eqref{equ:5.2} and enhance the practicality of the CS-mADBCD method, we explore the convergence rate and numerical performance of the CS-mADBCD method when $d=O(n)$.
\par
Analyzing the data in Tables \ref{tab:10}-\ref{tab:13}, we can draw the following conclusions. Firstly, regardless of whether the coefficient matrix of problem \eqref{linear least} is dense or sparse, the mADBCD method requires fewer iteration steps than the CS-mADBCD method. However, the former demands more CPU time than the latter. In comparison to the number of iteration steps, we are more concerned with CPU time because it represents the efficiency of a method. Therefore, the numerical performance of the CS-mADBCD method is superior to that of the mADBCD method. Secondly, we observe that as $d$ increases, the iteration steps of the CS-mADBCD method gradually decrease, while the values of CPU$_1$ and CPU$_2$ gradually increase, consequently causing an increase in the total time CPU for this method. Finally, the speed-ups is between 1.65 and 5.53 for dense matrices (refer to Tables \ref{tab:10}-\ref{tab:11}), and between 2.09 and 5.12 for sparse matrices (refer to Tables \ref{tab:12}-\ref{tab:13}).
  \begin{table}[!htbp]
    \centering
    \caption{IT and CPU of mADBCD and CS-mADBCD for different $\mathbf{A}$=\textbf{randn}$(400000,n)$.}\label{tab:10}%
    \begin{tabular*}{\hsize}{@{}@{\extracolsep{\fill}}|c|c|c|c|c|c|c|c|c|c|@{}}
      \hline 
       & \multirow[b]{2}{*}{(d, $\beta$, $\beta_{CS}$)} & \multicolumn{2}{l|}{ mADBCD  } & \multicolumn{5}{l|}{CS-mADBCD} \\      \hline 
     $n$ & & IT & CPU & IT & CPU$_1$ & CPU$_2$  & CPU & speed-up \\      \hline 
      \multirow[t]{4}{*}{$ 500$} & $(2n, 0, 0.55)$& $8$ & $3.0138
      $ & $38$ & $ 0.6442$ & $0.0048$ & $0.6490$ & $4.64$\\
      & $(4n, 0, 0.30)$ & $8$ & $3.7594$ & $18$ & $ 0.6733$ & $ 0.0062$ & $0.6794$ & $5.53$\\
      & $(8n, 0, 0.20)$ & $8$ & $3.0301$ & $13$ & $ 0.7272$ & $0.0162$ & $0.7434$ & $4.08$\\
      & $(12n, 0, 0.15)$ & $8$ & $3.0087$ & $12$ & $0.7371$ & $ 0.0238$ & $ 0.7610$ & $4.08$\\
      & $(20n, 0, 0.15)$ & $8$ & $ 3.2559 $ & $11$ & $ 0.8441$ & $0.0438$ & $0.8878$ & $3.67$\\
      \hline \multirow[t]{4}{*}{$1000$}& $(2n,0, 0.55)$& $8$ & $ 4.0817$ & $37$ & $ 1.6073$ & $ 0.0555$ & $ 1.6628$ & $2.45$\\
      & $(4n,0, 0.30)$ & $8$ & $4.1216$ & $18$ & $1.6341$ & $ 0.0623$ & $1.6964$ & $2.43$\\
      & $(8n, 0, 0.15)$ & $8$ & $ 4.5284$ & $13$ & $1.6623$ & $ 0.0869$ & $ 1.7491$ & $2.59$\\
      & $(12n, 0, 0.15)$ & $8$ & $4.0582$ & $12$ & $ 1.7006$ & $0.1151$ & $1.8157$ & $2.24$\\
      & $(20n, 0, 0.15)$ & $8$ & $ 4.1066$ & $11$ & $1.7504$ & $ 0.1700$ & $1.9203$ & $2.14$\\
      \hline \multirow[t]{4}{*}{$2000$}& $(2n, 0, 0.60)$& $9$ & $10.3952$ & $38$ & $3.5714$ & $0.2708$ & $3.8421$ & $2.71$\\
      & $(4n,0, 0.35)$ & $9$ & $ 11.3374$ & $19$ & $3.7427$ & $ 0.2769$ & $4.0196$ & $2.82$\\
      & $(8n, 0, 0.15)$ & $9$ & $ 10.5053$ & $13$ & $4.4646$ & $0.4071$ & $4.8717$ & $2.16$\\
      & $(12n, 0, 0.15)$ & $9$ & $11.0637$ & $12$ & $5.1079$ & $ 0.5862$ & $ 5.6941$ & $1.94$\\
      & $(20n, 0, 0.15)$ & $9$ & $10.2883$ & $11$ & $ 5.4775$ & $0.7610$ & $6.2385$ & $1.65$\\      \hline
      \end{tabular*}    
  \end{table}
  
  \begin{table}[!htbp]
    \centering
    \caption{IT and CPU of mADBCD and CS-mADBCD for different $\mathbf{A}$=\textbf{randn}$(800000,n)$.}\label{tab:11}%
    \begin{tabular*}{\hsize}{@{}@{\extracolsep{\fill}}|c|c|c|c|c|c|c|c|c|c|@{}}
      \hline   & \multirow[b]{2}{*}{(d, $\beta$, $\beta_{CS}$)} & \multicolumn{2}{l|}{ mADBCD  } & \multicolumn{5}{l|}{ CS-mADBCD} \\
      \hline $n$ & & IT & CPU & IT & CPU$_1$ & CPU$_2$  & CPU & speed-up \\
      \hline \multirow[t]{4}{*}{$200$} & $(2n, 0, 0.45)$& $8$ & $ 1.2379$ & $34$ & $0.5334$ & $0.0009$ & $0.5343$ & $2.32$\\
      & $(4n, 0, 0.30)$ & $7$ & $1.1168$ & $17$ & $0.5358$ & $0.0012$ & $ 0.5370$ & $2.08$\\
      & $(8n, 0, 0.20)$ & $7$ & $ 1.2029$ & $13$ & $ 0.5431$ & $ 0.0014$ & $ 0.5445$ & $2.21$\\
      & $(12n, 0, 0.15)$ & $7$ & $1.2352$ & $12$ & $0.5529$ & $0.0017$ & $ 0.5546$ & $2.23$\\
      & $(20n, 0, 0.15)$ & $8$ & $1.2046$ & $11$ & $  0.5642$ & $0.0022$ & $ 0.5665$ & $2.13$\\
      \hline \multirow[t]{4}{*}{$400$}& $(2n,0, 0.55)$& $8$ & $2.6605$ & $34$ & $ 1.1782$ & $ 0.0030$ & $1.1812$ & $2.25$\\
      & $(4n, 0, 0.30)$ & $8$ & $2.6878$ & $17$ & $ 1.1913$ & $ 0.0031$ & $1.1944$ & $2.25$\\
      & $(8n, 0, 0.20)$ & $8$ & $2.8013$ & $13$ & $1.2527$ & $0.0126$ & $1.2653$ & $2.21$\\
      & $(12n, 0, 0.20)$ & $8$ & $2.6497$ & $12$ & $1.2981$ & $ 0.0181 $ & $ 1.3163$ & $2.01$\\
      & $(20n, 0, 0.10)$ & $8$ & $2.7764$ & $11$ & $1.3347$ & $0.0288$ & $ 1.3635$ & $2.04$\\
      \hline \multirow[t]{4}{*}{$800$}& $(2n, 0, 0.60)$& $8$ & $ 6.1139$ & $37$ & $2.8606$ & $0.0326$ & $ 2.8932$ & $2.11$\\
      & $(4n, 0, 0.30)$ & $8$ & $6.2892$ & $18$ & $ 2.9666$ & $0.0400$ & $3.0066$ & $2.09$\\
      & $(8n, 0, 0.15)$ & $8$ & $5.5875$ & $13$ & $2.9733$ & $0.0649$ & $3.0382$ & $1.84$\\
      & $(12n, 0, 0.15)$ & $8$ & $6.3841$ & $12$ & $2.9739$ & $0.0860$ & $3.0598$ & $2.09$\\
      & $(20n, 0, 0.15)$ & $8$ & $6.1772$ & $11$ & $3.0938$ & $0.1312$ & $3.2250$ & $1.92$\\
      \hline
      \end{tabular*}    
  \end{table}

  \begin{table}[!htbp]
    \centering
    \caption{IT and CPU of mADBCD and CS-mADBCD for different $\mathbf{A}=$\textbf{sprandn}$(250000, n, 0.15)$.}\label{tab:12}%
    \begin{tabular*}{\hsize}{@{}@{\extracolsep{\fill}}|c|c|c|c|c|c|c|c|c|c|@{}}
      \hline  & \multirow[b]{2}{*}{(d, $\beta$, $\beta_{CS}$)} & \multicolumn{2}{l|}{ mADBCD  } & \multicolumn{5}{l|}{ CS-mADBCD} \\
      \hline $n$ & & IT & CPU & IT & CPU$_1$ & CPU$_2$  & CPU & speed-up \\
      \hline \multirow[t]{4}{*}{$250$} & $(2n, 0, 0.55)$& $8$ & $2.1190$ & $35$ & $0.8693$ & $0.0012$ & $0.8705$ & $2.43$\\
      & $(4n, 0, 0.30)$& $8$ & $2.2894$ & $18$ & $0.8842$ & $ 0.0011$ & $ 0.8853$ & $2.59$\\
      & $(8n, 0, 0.15)$& $8$ & $2.1884$ & $13$ & $0.8888$ & $0.0015$ & $  0.8903$ & $2.46$\\
      & $(12n, 0, 0.15)$& $8$ & $2.1781$ & $12$ & $ 0.9097$ & $ 0.0020$ & $0.9117$ & $2.39$\\
      & $(20n, 0, 0.05)$& $8$ & $ 2.1975$ & $10$ & $ 0.9150$ & $0.0045$ & $ 0.9195$ & $2.39$\\
      \hline \multirow[t]{4}{*}{$500$}& $(2n,0, 0.50)$& $8$ & $5.0271$ & $38$ & $1.3963$ & $0.0036$ & $1.3999$ & $3.59$\\
      & $(4n,0, 0.30)$& $8$ & $ 5.2635$ & $18$ & $ 1.3725$ & $0.0059$ & $ 1.3784$ & $3.82$\\
      & $(8n, 0, 0.20)$& $8$ & $5.8286$ & $13$ & $ 1.4704$ & $ 0.0169$ & $1.4873$ & $3.92$\\
      & $(12n, 0, 0.15)$& $8$ & $5.3982$ & $12$ & $1.5076$ & $0.0296$ & $1.5372$ & $3.51$\\
      & $(20n, 0, 0.15)$& $8$ & $5.1187$ & $11$ & $1.5642$ & $0.0404$ & $1.6046$ & $3.19$\\
      \hline \multirow[t]{4}{*}{$1000$}& $(2n,0, 0.55)$& $9$ & $ 13.0300$ & $36$ & $ 2.5013$ & $ 0.0425$ & $ 2.5438$ & $5.12$\\
      & $(4n,0, 0.35)$& $9$ & $14.0215$ & $19$ & $ 2.8323$ & $0.0685$ & $ 2.9008$ & $4.83$\\
      & $(8n, 0, 0.15)$& $9$ & $ 12.7943$ & $13$ & $ 2.8739$ & $0.0827$ & $2.9565$ & $4.33$\\
      & $(12n, 0, 0.15)$& $8$ & $11.4808$ & $12$ & $ 2.9394$ & $ 0.1272$ & $3.0666$ & $3.74$\\
      & $(20n, 0, 0.15)$& $8$ & $ 11.3291 $ & $11$ & $2.9854$ & $  0.2131$ & $ 3.0385
      $ & $3.73$\\
      \hline
      \end{tabular*}    
  \end{table}

  \begin{table}[!htbp]
    \centering
    \caption{IT and CPU of mADBCD and CS-mADBCD for different $\mathbf{A}=$\textbf{sprandn}$(500000, n, 0.075)$.}\label{tab:13}%
    \begin{tabular*}{\hsize}{@{}@{\extracolsep{\fill}}|c|c|c|c|c|c|c|c|c|c|@{}}
      \hline & \multirow[b]{2}{*}{(d, $\beta$, $\beta_{CS}$)} & \multicolumn{2}{l|}{ mADBCD  } & \multicolumn{5}{l|}{ CS-mADBCD} \\
      \hline $n$ & & IT & CPU & IT & CPU$_1$ & CPU$_2$  & CPU & speed-up \\
      \hline \multirow[t]{4}{*}{$500$} & $(2n, 0, 0.55)$& $8$ & $4.7912
      $ & $37$ & $ 2.1666$ & $0.0044$ & $2.1710$ & $2.21$\\
      & $(4n, 0, 0.30)$& $8$ & $4.9453$ & $18$ & $2.1995$ & $ 0.0055$ & $2.2050$ & $2.24$\\
      & $(8n, 0, 0.15)$& $8$ & $4.8489$ & $13$ & $2.2574$ & $0.0172$ & $2.2746$ & $2.13$\\
      & $(12n, 0, 0.15)$& $8$ & $ 5.3409$ & $12$ & $ 2.3166$ & $0.0253$ & $2.3419$ & $2.28$\\
      & $(20n, 0, 0.10)$& $8$ & $4.8892$ & $11$ & $2.2884$ & $0.0467$ & $2.3351$ & $2.09$\\
      \hline \multirow[t]{4}{*}{$1000$}& $(2n,0, 0.55)$& $8$ & $ 11.3653$ & $36$ & $3.6586$ & $ 0.0437$ & $3.7023$ & $3.07$\\
      & $(4n,0, 0.30)$& $8$ & $11.3695$ & $18$ & $ 3.6373 $ & $0.0543$ & $3.6915$ & $3.08$\\
      & $(8n, 0, 0.20)$& $8$ & $11.5792$ & $13$ & $3.7178 $ & $0.0791$ & $3.7969$ & $3.05$\\
      & $(12n, 0, 0.15)$& $8$ & $10.8494$ & $12$ & $3.8236$ & $0.1079 $ & $ 3.9315$ & $2.76$\\
      & $(20n, 0, 0.15)$& $8$ & $11.4601$ & $11$ & $4.1175 $ & $0.1933$ & $ 4.3108$ & $2.66$\\
      \hline \multirow[t]{4}{*}{$2000$}& $(2n,0, 0.55)$& $8$ & $27.4519$ & $37$ & $7.3132$ & $0.2843$ & $7.5976$ & $3.61$\\
      & $(4n,0, 0.30)$& $8$ & $32.7344$ & $18$ & $7.3642 $ & $0.2563$ & $7.6206$ & $4.30$\\
      & $(8n, 0, 0.15)$& $8$ & $ 29.7545$ & $13$ & $7.4427$ & $ 0.3277$ & $ 7.7704$ & $3.83$\\
      & $(12n, 0, 0.15)$& $8$ & $27.2784$ & $12$ & $7.5501 $ & $  0.5225$ & $ 8.0726$ & $3.38$\\
      & $(20n, 0, 0.15)$& $8$ & $27.4581$ & $11$ & $8.0427$ & $  0.8337$ & $8.8763$ & $3.09$\\
      \hline
      \end{tabular*}    
  \end{table}

\section{Conclusion}\label{sec:6}
In this paper, we propose an adaptive block coordinate descent method with momentum. In similar manner to the GBGS, MRBGS and FBCD methods \cite{ELiZ,ELiLuX,EChHu}, the mADBCD method uses all the elements in the indexed set per iteration to update the next iteration vector, which is different from the RCD, GRCD and RGRCD methods \cite{ELeLa,EBaWu,EZhG} that use only one element to get the next approximate solution. Under the condition of full rank columns in the coefficient matrix, theoretical analysis reveals that the method converges to the unique solution of the linear least-squares problem. It also indicates that the maximum convergence rate of the mADBCD method is closely related to the minimum singular value of the coefficient matrix, the maximum singular value of the selected submatrix, and the momentum parameter $\beta$. Secondly, we effectively integrate count sketch technology with the mADBCD method, designing a more efficient algorithm for solving extremely overdetermined least-squares problems, namely, the CS-mADBCD method, and demonstrate its convergence. In addition, the effect of the momentum parameter $\beta$ on the numerical performance of the mADBCD method for solving linear least-squares problems has also been investigated by numerical experiments. And we find that for larger values of $\frac{m}{n}$, the relatively suitable parameter $\beta$ of the method takes a smaller value, particularly, in Table \ref{tab:10}-\ref{tab:13}, it can be seen that for extremely overdetermined least-squares problems, the optimal values of $\beta$ for the mADBCD method are all 0. Finally, some numerical experiments demonstrate that the proposed method is more effective in solving problem \eqref{linear least} compared to the GBGS, MRBGS and FBCD methods.
\par
From Theorems \ref{theo:3.1} and \ref{theo:4.1}, it can be observed that the optimal momentum parameter values of $\beta$, used to minimize the convergence speed upper bounds for the mADBCD and CS-mADBCD methods, are closely related to the properties of the coefficient matrices and the selected column submatrices at each iteration. Therefore, it is currently a challenging task theoretically to determine the optimal values or range of momentum parameter $\beta$ for achieving the best convergence speed for the mADBCD and CS-mADBCD methods. Hence, designing an efficient block coordinate descent method with adaptively selected favorable momentum parameters remains to be explored and is one of our research directions in the future.

\bmhead{Funding}This work is partly supported by National Natural Science Foundation of China (No. 12071149), Science and Technology Commission of Shanghai Municipality(No. 22DZ2229014), National Key Research and Development Program(No. 2022YFA1004403).

\bmhead{Data Availability Statement}Data sharing not applicable to this article as no datasets were generated or analyzed during the current study.

\section*{Declarations}
\bmhead{Ethics approval} Not applicable.
\bmhead{Competing interests} The authors declare no competing interests.

\bibliography{sn-bibliography}

\begin{thebibliography}{99}
  \bibitem{EAmDis}Arioli M, Duff, I.S.: Preconditioning linear least-squares problems by identifying a basis matrix. SIAM J. Sci. Comput. 37(5), S544--S561 (2015).

  \bibitem{EBW4}Bai, Z.-Z., Wu, W.-T.: On relaxed greedy randomized kaczmarz methods for solving large sparse linear systems. Appl. Math. Lett. 83, 21--26 (2018).

  \bibitem{EBaWu} Bai, Z.-Z., Wu, W.-T.: On greedy randomized coordinate descent methods for solving large linear least-squares problems. Numer. Linear Algebra Appl. 26(4), e2237 1--15 (2019).

  \bibitem{EBjo}Bj{\"o}rck, A.: Numerical methods for least squares problems. SIAM, Philadelphia., 1996.

  \bibitem{EChCFC}Charikar, M., Chen, K., Farach-Colton, M.: Finding frequent items in data streams. In Automata Lang. Pragram. 693--703, (2002).

  \bibitem{EChHu}Chen, J.-Q., Huang, Z.-D.: A Fast Block Coordinate Descent Method for Solving Linear Least-Squares Problems. East Asian J. Appl. Math. 42(7), 406-420, (2022).

  \bibitem{EDaHu}Davis, T.,  Hu. Y.: The university of florida sparse matrix collection. ACM Trans. Math. Software 38 (1) (2011) 1--25.


  \bibitem{EDMMSa}Drineas, P., Mahoney, M.W., Muthukrishnan, S, Sarls, T.: Faster least squares approximation. Numer math. 117(2), 219--249 (2011).

  \bibitem{EDu}Du, K.: Tight upper bounds for the convergence of the randomized extended Kaczmarz and Gauss-Seidel algorithms. Numer. Linear Algebra Appl. 26(3), e2233, 1--14 (2019).

  \bibitem{EEMZ}Elad, M., Matalon, B., Zibulevsky, M.: Coordinate and subspace optimization methods for linear least squares with non-quadratic regularization. Appl. Comput. Harmonic Anal. 23(3), 346--367 (2007).

  \bibitem{EHaJo}Hansen, P.C., Jorgensen, J.S.: AIR tools II: algebraic iterative reconstruction methods, improved implementation. Numer. Algorithms 79(1), 107--137 (2018).

  \bibitem{EHX}Han, D.-R., Xie J.-X.: On pseudoinverse-free randomized methods for linear systems: Unified framework and acceleration. arXiv preprint arXiv:2208.05437, (2022).


  \bibitem{EHeNR}Hefny, A., Needell, D., Ramdas, A.: Rows versus columns: randomized Kaczmarz or Gauss-Seidel for ridge regression. SIAM J. Sci. Comput. 39(5), S528--S542 (2017).

  \bibitem{EHinj}Higham, N.J.: Accuracy and stability of numerical algorithms. SIAM, Philadelphia, 2002.

  \bibitem{EHWMAB}Hoyos-Idrobo, A., Weiss, P., Massire, A., Amadon, A., Boulant, N.: On variant strategies to solve the magnitude least squares optimization problem in parallel transmission pulse design and under strict SAR and power constraints. IEEE Trans. Med. Imaging 33, 739--748 (2014).

  \bibitem{ELeLa}Leventhal, D., Lewis, A.S.: Randomized methods for linear constraints: convergence rates and conditioning. Math. Oper. Res. 35(3), 641--654 (2010).

  \bibitem{ELiLuX}Liu, Q., Lu, Z., Xiao, L.: An accelerated randomized proximal coordinate gradient method and its application to regularized empirical risk minimization. SIAM J. Optim. 25(4), 2244--2273 (2015).

  \bibitem{ELiJG}Liu, Y., Jiang, X.-L., Gu, C.-Q.: On maximum residual block and two-step Gauss-Seidel algorithms for linear least-squares problems. Calcolo 58, 13, 1-32 (2021).

  \bibitem{ELiZ}Li, H.-Y., Zhang, Y.-J.: Greedy block Gauss-Seidel methods for solving large linear least squares problem. arXiv preprint arXiv:2004.02476v1 (2020).

  \bibitem{EMaNR}Ma, A., Needell, D., Ramdas, A.: Convergence properties of the randomized extended Gauss--Seidel and Kaczmarz methods. SIAM J. Matrix Anal. Appl. 36(4), 1590--1604 (2015).

  \bibitem{ENeNeG}Necoara, I., Nesterov, Y.,  Glineur, F.: Random Block Coordinate Descent Methods for Linearly Constrained Optimization over Networks. J Optim Theory Appl 173, 227--254 (2017).

  \bibitem{ENeeT}Needell, D., Tropp, J.A.: Paved with good intentions: analysis of a randomized block Kaczmarz method. Linear Algebra Appl. 441, 199--221 (2014).

  \bibitem{ENeZhZ}Needell, D., Zhao, R., Zouzias, A.: Randomized block Kaczmarz method with projection for solving least squares, Linear Algebra Appl. 484, 322--343 (2015).

  \bibitem{ENes}Nesterov, Y.: Efficiency of coordinate descent methods on huge-scale optimization problems. SIAM J. Optim. 22(2), 341--362 (2012)

  \bibitem{EPoBT}Polyak, B.T.: Some methods of speeding up the convergence of iteration methods. Comput. Math. Math. Phys., 4(5):1--17, 1964.


  \bibitem{ERa}Ruhe, A.: Numerical aspects of Gram-Schmidt orthogonalization of vectors. Linear Algebra Appl. 52, 591--601 (1983).

  \bibitem{ESal}Sarls, T.: Improved approximation algorithms for large matrices via random projections. 47th Annual IEEE Sympos. Found. Comput. Sci. (FOCS'06), 143--152, (2006).

  \bibitem{EST}Scott, J.A., Tuma, M.: Sparse stretching for solving sparse-dense linear least-squares problems. SIAM J. Sci. Comput. 41(3), A1604--A1625 (2019).

  \bibitem{ETaGo} Tan, L.-Z, Guo, X.-P.: On multi-step greedy randomized coordinate descent method for solving large linear least-squares problems. Comp. Appl. Math. 42(1), 37, 1--20 (2023).

  \bibitem{EThZh}Thorup, M., Zhang, Y.: Tabulation-based 5-independent hashing with applications to linear probing and second moment estimation. SIAM J. Comput. 41(2), 293--331. (2012).

  \bibitem{EWWY}Wang, C., Wu, D., Yang, K.: New decentralized positioning schemes for wireless sensor networks based on recursive least-squares optimization. IEEE Wirel. Commun. Lett. 3, 78--81 (2014).

  \bibitem{EWood}Woodruff, D.P.: Sketching as a tool for numerical linear algebra, Found. Trends Theor. Comput. Sci. 10 (1-2), 1--157, (2014).

  \bibitem{EWuNe}Wu, W., Needell, D.: Convergence of the randomized block Gauss-Seidel method, SIAM Undergraduate Res. Online, 11, 369--382 (2018).

  \bibitem{EZhG}Zhang, J.-H., Guo, J.-H.: On relaxed greedy randomized coordinate descent methods for solving large linear least-squares problems, Appl. Numer. Math., 157, 372--384, (2020).

  \bibitem{EZhLi}Zhang, Y.-J., Li, H.-Y.: A novel greedy Gauss-Seidel method for solving large linear least squares problem. arXiv preprint arXiv: 2004.03692v1 (2020).

  \bibitem{EZhLih}Zhang, Y.-J. and Li, H.-Y.: A Count Sketch Maximal Weighted Residual Kaczmarz Method for Solving Highly Overdetermined Liner Systems. Applied Mathematics and Computation, 410, 126486, 1--7, (2021).
\end{thebibliography}

\end{document}